\documentclass[a4paper,10pt]{amsart}

\usepackage{amsmath}
\usepackage{amssymb}
\usepackage{amsthm}
\usepackage[a4paper]{geometry}
\usepackage{todonotes}
\usepackage{enumerate}   
\usepackage{cite}
\usepackage{enumitem}
\usepackage{hyperref}

\theoremstyle{plain}
\newtheorem{thm}{Theorem}[section]
\newtheorem*{thm*}{Theorem}
\newtheorem{cor}{Corollary}[thm]
\newtheorem{lem}[thm]{Lemma}
\newtheorem{prop}[thm]{Proposition}

\theoremstyle{definition}
\newtheorem{defn}{Definition}[section]
\newtheorem{exmp}{Example}[section]

\theoremstyle{remark}
\newtheorem*{rem}{Remark}

\newcommand\inner[2]{\langle #1, #2 \rangle}

\newcommand{\tco}{\mathcal{T}^1}
\newcommand{\bo}{B(L^2(\R^d))}
\newcommand{\SC}{\mathcal{T}}

\newcommand{\C}{\mathbb{C}}
\newcommand{\N}{\mathbb{N}}
\newcommand{\R}{\mathbb{R}}
\newcommand{\Rd}{\mathbb{R}^d}
\newcommand{\Rdd}{\mathbb{R}^{2d}}

\newcommand{\Ldd}{L^1(\Rdd)}

\newcommand{\HS}{L^2(\R^d)}
\newcommand{\F}{\mathcal{F}}
\newcommand{\tr}{\mathrm{tr}}
\newcommand{\schwartz}{\mathfrak{S}}
\newcommand{\tempdist}{\mathfrak{S}^{\prime}}

\begin{document}
%%%%%%%%%%%%%%%%%%%%%%%%%%%%%%%%%%%%%%%%%%%%%%%%%%%%%%%%%%%%%%%%%%%%%%%%
\pagestyle{plain}
\title{Mixed-state localization operators: Cohen's class and trace class operators}
\author{Franz Luef}
\author{Eirik Skrettingland} 
\address{Department of Mathematics\\ NTNU Norwegian University of Science and
Technology\\ NO–7491 Trondheim\\Norway}
\email{franz.luef@math.ntnu.no, eirik.skrettingland@ntnu.no}
%  Math Subject Classifications
%%%%%%%%%%%%%%%%%%%%%%%%%%%%%%%%%%%%%%%%%%%%%%%%%%%%%%%%%%%%%%%%%%%%%%%%%%%%%%%%%%%%%%%%%%%%%%%%%%%%%%%%%%%%%%%%%%%%%%%%%%%%%%%%%%%%%%%%%%%%%%%%%
\begin{abstract}
We study mixed-state localization operators from the perspective of Werner's operator convolutions which allows us to extend known results from the rank-one case to trace class operators. The idea of localizing a signal to a domain in phase space is approached from various directions such as bounds on the spreading function, probability densities associated to mixed-state localization operators, positive operator valued measures, positive correspondence rules and variants of Tauberian theorems for operator translates. Our results include a rigorous treatment of multiwindow-STFT filters and a characterization of mixed-state localization operators as positive correspondence rules. Furthermore we provide a description of the Cohen class in terms of Werner's convolution of operators and deduce consequences on positive Cohen class distributions, an uncertainty principle, uniqueness and phase retrieval for general elements of Cohen's class.  
\end{abstract}
\maketitle \pagestyle{myheadings} \markboth{F. Luef and E. Skrettingland}{Mixed-state localization operators}
\thispagestyle{empty}
%%%%%%%%%%%%%%%%%%%%%%%%%%%%%%%%%%%%%%%%%%%%%%%%%%%%%%%%%%%%%%%%%%%%%%%%%%%%%%%%%%%%%%%%%%%%%%%%%%%%%%%%%%%%%%%%%%%%%%%%%%%%%%%%%%%%%%%%%%%%%%%%%
\section{Introduction}

We are addressing some key problems of time-frequency analysis: (i) How to measure the time-frequency content of a signal? (ii)  What is the 
effect a (linear) filter has on a signal? Over the years engineers and mathematicians have investigated these questions and have proposed a variety of 
answers as is demonstrated by the vast literature \cite{bocogr04,Boggiatto:2010,Boggiatto:2017,Cordero:2003,Cordero:2005,da88,Flandrin:1988,Janssen:1997,Ramanathan:1994,Ramanathan:1994spec}. We approach these problems from the perspective of quantum harmonic analysis and 
note that notions and results in \cite{Werner:1984} provide a unifying umbrella for some of the research in this direction such as localization operators, multiwindow STFT-filters, Cohen's class of quadratic time-frequency representations and the spreading function of a filter.
  
Harmonic analysis is based on the interplay between the translation of a function, convolution of functions and the Fourier transform. In \cite{Werner:1984} analogues of these notions are introduced for operators:  The {\it translation} of an operator $A$ by a point $z=(x,\omega)$ in phase space $\Rdd$ is defined by conjugation with the time-frequency shift $\pi(z)$:
\begin{equation*}
  \alpha_z(A)=\pi(z)A\pi(z)^*,
\end{equation*}
where $\pi(z)\psi(t)=e^{2\pi i\omega t}\psi(t-x)$. In \cite{Luef:2017vs} we showed that this yields a natural class of Banach modules. There are two types of convolutions in this noncommutative setting: (i) The convolution between a function  $f\in L^1(\R^{2d})$ 
and a trace class operator  $S$: 
\begin{equation*}
  f\star S := S\star f := \iint_{\R^{2d}}f(y)\alpha_y(S) \ dy;
\end{equation*}
(ii) the convolution between two trace class operators $S$ and $T$ is defined by 
\begin{equation*}
  S \star T(z) = \tr(S\alpha_z (\check{T}))
\end{equation*}
for $z\in \R^{2d}$, where $\check{T}=PTP$ is defined by conjugation by the parity operator $P$. Finally, the analogue of the Fourier transform is given by the {\it Fourier-Wigner transform}
$\F_WS$ of a trace class operator $S$, which is the function given by
	\begin{equation*}
	\F_W S(z)=e^{-\pi i x \cdot \omega}\tr(\pi(-z)S)
	\end{equation*}
	for $z\in \R^{2d}$. Note that the Fourier-Wigner transform and the spreading function differ only by a phase factor. The Fourier-Wigner transform has many properties analogous to those of the Fourier transform of functions\cite{Werner:1984, Luef:2017vs}.
\\
	In the case of rank-one operators these concepts of quantum harmonic analysis turn into well-known objects from time-frequency analysis. Suppose $\varphi_2 \otimes \varphi_1$ is the rank-one operator for 
	$\varphi_1,\varphi_2\in L^2(\Rd)$. Then we have 
	\[f\star (\varphi_2 \otimes \varphi_1)=\iint_{\mathbb{R}^{2d}}f(z)V_{\phi_1}\psi(z)\pi(z)\phi_2\,dz, \]
which is a localization operator (or STFT-filter or STFT-multiplier \cite{Kozek:1992,feno03}) and is denoted by $\mathcal{A}_f^{\varphi_1,\varphi_2}$, and $f$ is called the mask of the STFT-filter. Similarly, the convolution of two rank-one operators becomes 
\[(\phi\otimes \psi)\star (\check{\xi} \otimes \check{\eta})(z)= V_{\eta}\phi(z) \overline{V_{\xi}\psi(z)},\] 
where $\check{\xi}(x)=\xi(-x)$, which reduces for $\eta=\psi$ and $\psi=\phi$ to the spectrogram \cite{Keller:2017}. The Fourier-Wigner transform of a rank-one operator is the ambiguity function. There is also a Hausdorff-Young inequality associated to the Fourier-Wigner transform \cite{Werner:1984, Luef:2017vs}, that in the rank-one case is the non-sharp Lieb's inequality for ambiguity functions \cite{li90-1}. 
\\
Let us return to the objectives of this paper. Since localization operators are convolutions of a function and a rank-one operator, a natural extension of localization operators are operators of the form $f\star S$ for a trace-class operator $S$. The results of this paper indicate that these operators describe the time-frequency localization in various ways. For example we are interested in the amount of "spreading" in time and frequency that an operator performs on a function which we describe in form of bounds on the concentration of the spreading function, or equivalently of its Fourier-Wigner transform. The next theorem is an example for the type of statements we have in mind: 
\begin{thm*}
Let $S$ be a trace-class operator and let $\Omega\subset \R^{2d}$ with $\mu(\Omega)<\infty$ and assume that
	\begin{equation*}
  \iint_{\Omega} |\F_W(S)(z)|^2 \ dz \geq 1-\epsilon
\end{equation*}
for some $\epsilon \geq 0$. For any $p>2$ we then have
\begin{equation*}
  \mu(\Omega)\geq \frac{(1-\epsilon)^{p/(p-2)} \left(\frac{p}{2}\right)^{2d/(p-2)}}{\|S\|_{\tco}^{2p/(p-2)}},
\end{equation*}
where $\|S\|_{\tco}$ denotes the trace class norm of $S$. 
\end{thm*}
One interpretation of this uncertainty principle is that a well-concentrated spreading function comes at the cost of a large trace class norm.	The proof is a consequence of the Hausdorff--Young inequality for the Fourier-Wigner transform of $S$.
\\
In the engineering literature \cite{Kozek:1992,Hlawatsch:1992} one calls an operator 
\begin{equation*}
 H= \sum_{n} \lambda_n \mathcal{A}_f^{\varphi_{n,1},\varphi_{n,2}}
\end{equation*}
a \textit{multiwindow STFT-filter}, where $\{ \lambda_n \}_{n\in \N}$ is a sequence of complex numbers and $\{\varphi_{n,1}\}_{n\in \N}$ and $\{\varphi_{n,2}\}_{n\in \N}$ are sequences of functions in $L^2(\R^d)$, \cite{Kozek:1992}. Multiwindow STFT filters might be thought of as an operator that change the signal by some smearing. We give a rigorous treatment of the boundedness of multiwindow STFT filters depending on the sequence $(\lambda_n)_{n\in\N}$ and prove that multiwindow STFT-filters are given by a function convolved with an operator. 
\\
In addition we consider the set of multiwindow STFT-filters $f\star S$ for functions $f$ for a fixed trace-class operator $S$. Using the Tauberian theorem for convolutions with operators (theorem \ref{thm:wiener}), we are able to show (under some assumptions on the Fourier-Wigner spectrum): (i) any Schatten class operator $T$ may be approximated by operators of the form $f\star S$; (2) that the mask $f$ is uniquely determined by the operator $f\star S$. As a sample we have results of the following form: For a trace-class operator $S$ the following are equivalent:
	\begin{enumerate}
		\item The set $\{z\in \R^{2d}:\F_W(S)=0\}$ is empty.
		\item The set of multiwindow STFT-filters $ f\star S$ with $f\in L^1(\Rdd)$ is dense in the set of trace-class operators.
		\item Any mask $f \in L^{\infty}(\mathbb{R}^{2d})$ is uniquely determined by the multiwindow STFT-filter $f\star S$.
	\end{enumerate}
 In order to gain some understanding of the notion of localization in this context, we focus on operators $H_{\Omega}$ of the form
\begin{equation*}
  H_{\Omega}= \chi_{\Omega}\star S
\end{equation*}
where $\chi_\Omega$ is the indicator function of a measurable subset $\Omega$ of $\Rdd$ and $S$ is a positive trace class operator with $\tr(S)=1$. We refer to these operators as \textit{mixed-state localization operators}. 
\\
Given a mixed-state localization operator $\chi_{\Omega}\star S$, one might ask whether it is possible to recover information about the domain $\Omega$ from the operator $\chi_{\Omega}\star S$. We show that the measure of $\Omega$ may be calculated from the eigenvalues of $\chi_{\Omega}\star S$ and  we also consider the problem of reconstructing the domain $\Omega$ from $H_\Omega$. Finally we also discuss in which sense an operator $H_\Omega$ measures the time-frequency content of a signal in a domain $\Omega$. These questions have received some attention \cite{Abreu:2012,Abreu:2016} in recent years. Our techniques provide a way to handle unbounded domains, which have not been treated previously in the literature.
\\
The treatment of mixed-state localization operators leads us to the investigation of Cohen class distributions \cite{Cohen:1966}. We show that any Cohen class distribution $Q_S(\psi)$ is of the form
\begin{equation*}
  Q_S(\psi)=(\psi\otimes \psi)\star \check{S},
\end{equation*} 
where $S$ is a trace-class operator. We establish an uncertainty principle for Cohen class distributions and ask whether any square-integrable function is uniquely determined by the associated Cohen class distribution, which in a special case was discussed in \cite[Remark A.4]{Grohs:2017} for the spectrogram. In addition we characterize when Cohen class distributions are positive and have the correct total energy properties.
\\
We observe also that mixed-state localization operators define positive operator valued measures (POVMs), a standard tool in quantum mechanics, see 
\cite{Moran:2013,Han:2014,Han:2014a} for some relations between POVMs and frame theory. By a theorem of Holevo \cite{Holevo:1979} of positive correspondence rules we have that this is in a sense the only way to produce (covariant) POVMs. We will also argue that the notion of POVM is a natural framework for localization operators and Cohen's class of time-frequency distributions and that a POVM allows one to construct a probability measure on phase space. This measure is absolutely continuous with respect to Lebesgue measure and its Radon-Nikodym derivative is a positive Cohen class distribution. 
\section{Notation and terminology}
If $X$ is a Banach space we will denote its dual space by $X^*$, and for $x\in X$ and $x^* \in X^*$ we write $\inner{x^*}{x}_{X^*,X}$ to denote $x^*(x)$. $\inner{\cdot}{\cdot}$ denotes the inner product on the Hilbert space $L^2(\Rd)$. Note that $\inner{\cdot}{\cdot}_{X^*,X}$ is bilinear, whereas $\inner{\cdot}{\cdot}$ is antilinear in the second argument.
Elements of $\R^{2d}$ will often be written in the form $z=(x,\omega)$ for $x,\omega\in \R^d$, and the Lebesgue measure of a subset $\Omega\subset \R^{2d}$ will be denoted by $\mu(\Omega)$. The characteristic function of $\Omega \subset \R^{2d}$ is denoted by $\chi_{\Omega}$. $\sigma(z,z^\prime)$ is the standard symplectic form $\sigma(z,z^\prime)=\omega_1\cdot x_2-\omega_2 \cdot x_1$ of $z=(x_1,\omega_1)$ and $z^\prime=(x_2, \omega_2)$. For two functions $\xi, \eta$ in the Hilbert space $L^2(\R^d)$, we define the operator $\xi \otimes \eta$ on $\HS$ by $\xi \otimes \eta (\zeta)=\inner{\zeta}{\eta}\xi$, where $\zeta \in \HS$. 
The space of Schwartz functions on $\R^{2d}$ is denoted by $\mathcal{S}(\R^{2d})$ and its dual space of tempered distributions by $\mathcal{S}^{\prime}(\R^{2d})$.
We introduce the parity operator $P$ by  $\check{\psi}(x)=P\psi(x)=\psi(-x)$ for any $x\in \R^d$ and $\psi:\R^d\to \mathbb{C}$, and define $\psi^*$ by $\psi^*(x)=\overline{\psi(x)}$. 
\section{Preliminaries}
\subsection{Concepts from time-frequency analysis}
\subsubsection{The symplectic Fourier transform}
For functions $f\in L^1(\R^{2d})$ we will use the \textit{symplectic Fourier transform} $\F_{\sigma} f$, given by 
\begin{equation*}
\F_{\sigma} f(z)=\iint_{\R^{2d}} f(z') e^{-2 \pi i \sigma(z,z')} \ dz'
\end{equation*}
for $z\in \R^{2d}$, where $\sigma$ is the standard symplectic form $\sigma((x_1,\omega_1),(x_2, \omega_2))=\omega_1\cdot x_2-\omega_2 \cdot x_1$. $\F_{\sigma}$ extends to a unitary operator on $L^2(\Rdd)$, and this extension satisfies $\F_{\sigma}^2=I$, where $I$ is the identity operator\cite{deGosson:2011wq}.
\subsubsection{The STFT, Wigner distribution and the Weyl calculus}
If $\psi:\R^d\to \mathbb{C}$ and $z=(x,\omega)\in \R^{2d}$, we define the \textit{translation operator} $T_x$ by $T_x\psi (t)=\psi(t-x)$, the \textit{modulation operator} $M_{\omega}$ by $M_{\omega}\psi (t)=e^{2 \pi i \omega \cdot t} \psi (t)$ and the \textit{time-frequency shifts} $\pi(z)$ by $\pi(z)=M_{\omega}T_x$. 
For $\psi,\phi \in L^2(\R^d)$ the \textit{short-time Fourier transform} (STFT) $V_{\phi}\psi$ of $\psi$ with window $\phi$ is the function on $\R^{2d}$
 defined by
 \begin{equation*}
  V_{\phi}\psi(z)=\inner{\psi}{\pi(z)\phi}
\end{equation*}
for $z\in \R^{2d}$. By replacing the inner product above with a duality bracket, the STFT may be extended to other spaces, such as $\psi\in \mathcal{S}(\Rd), \phi \in \mathcal{S}^{\prime}(\Rd)$.
We will also refer to the \textit{cross-ambiguity function} $A(\psi,\phi)$ of $\psi$ and $\phi$, defined by  multiplying the STFT with a phase factor:
$$
  A(\psi, \phi)(z)=e^{\pi i x \cdot \omega} V_{\phi}\psi(z).
$$
For more background on the ambiguity function and its utility in the theory of radar see \cite{Grochenig:2001,Folland:1989}. A close relative of the STFT is the \textit{cross-Wigner distribution} of two functions $\psi$ and $\phi$ on $\R^d$. By definition, the cross-Wigner distribution $W(\psi,\phi)$ is given by
\begin{equation*}
  W(\psi,\phi)(x,\omega)=\int_{\R^d} \psi\left(x+\frac{t}{2}\right)\overline{\phi\left(x-\frac{t}{2}\right)} e^{-2 \pi i \omega \cdot t} \ dt.
 \end{equation*}	
This expression is similar to the definition of the STFT, and in fact $W(\psi,\phi)=\F_{\sigma}(A(\psi,\phi))$ \cite{deGosson:2011wq}.

Using the cross-Wigner distribution, we may introduce the \textit{Weyl calculus}. For $f \in \mathcal{S}^{\prime}(\R^{2d})$ and $\psi,\phi \in \mathcal{S}(\R^d)$, we define the \textit{Weyl transform} $L_{f}$ of $f$ to be the operator given by 
\begin{equation*}
  \inner{L_{f}\psi}{\phi}_{\mathcal{S}^\prime,\mathcal{S}}=\inner{f}{W(\psi,\phi)}_{\mathcal{S}^\prime,\mathcal{S}}.
\end{equation*}

$f$ is called the \textit{Weyl symbol} of the operator $L_{f}$.
\subsubsection{Cohen's class of quadratic time-frequency distributions} \label{sec:cohensclass}
A quadratic time-frequency distribution $Q$ is said to be of \textit{Cohen's class} if $Q$ is given by 
\begin{equation*}
  Q(\psi)=Q_{\phi}(\psi):=W(\psi,\psi)\ast \phi
\end{equation*}
for some $\phi \in \mathcal{S}^{\prime}(\R^{2d})$ \cite{Cohen:1966,Grochenig:2001}. The class of functions $\psi$ to which we may apply $Q_{\phi}$ clearly depends on the distribution $\phi$. The Wigner distribution is obtained by picking $\phi=\delta_0$, where $\delta_0$ is Dirac's delta distribution centered at $0$. Cohen's class contains all shift-invariant, weakly continuous quadratic time-frequency distributions, as is made precise by the following lemma from \cite[Thm. 4.5.1]{Grochenig:2001}.
\begin{lem} \label{lem:grochenig}
	Let $Q$ be a quadratic time-frequency distribution satisfying 
	\begin{enumerate}
		\item $Q(\pi(z)\psi)=T_z(Q(\psi)) $,
		\item $|Q(\psi_1,\psi_2)(0)|\leq \|\psi_1\|_2\|\psi_2\|_2$,
	\end{enumerate}
	for all $z\in \Rdd$ and $\psi_1,\psi_2 \in \HS$. Then $Q(\psi)=W(\psi,\psi)\ast \phi$ for some $\phi\in \mathcal{S}^{\prime}(\Rdd)$.
\end{lem}
\subsection{Concepts from operator theory}
\subsubsection{The Schatten classes of operators}
\label{sec:schatten}
 In classical harmonic analysis one often studies the $L^p$-spaces of functions, and we will similarly need to introduce classes of operators in $\bo$ with different properties. To introduce these classes, we need the \textit{singular value decomposition} of compact operators on $\HS$  \cite{Reed:1980}.
\begin{prop}
\label{prop:singval}
	Let $S$ be a compact operator on $\HS$. There exist two orthonormal sets $\{\psi_n\}_{n\in \mathbb{N}}$ and $\{\phi_n\}_{n\in \mathbb{N}}$ in $\HS$ and a sequence $\{s_n(S)\}_{n\in \mathbb{N}}$ of positive numbers with $s_n(S) \to 0$, such that $S$ may be expressed as 
\begin{equation*}
S = \sum\limits_{n \in \mathbb{N}} s_n(S) \psi_n\otimes \phi_n, 
\end{equation*}
with convergence in the operator norm.
The numbers $\{s_n(S)\}_{n\in \mathbb{N}}$ are called the \textit{singular values} of $S$, and are the eigenvalues of the operator $|S|$. 
\end{prop}

For $1\leq p<\infty$ we define the \textit{Schatten class} $\SC^p$ of operators by $$\SC^p=\lbrace T\text{ compact}: (s_n(T))_{n\in \mathbb{N}} \in \ell^p\rbrace.$$ 
	We will also write $\SC^{\infty}=\bo$ with $\|\cdot\|_{\SC^\infty}$ given by the operator norm to simplify the statement of some results. The Schatten class $\SC^p$ becomes a Banach space under pointwise addition and scalar multiplication in the norm $\|S\|_{\SC^p}=\left(\sum\limits_{n\in \mathbb{N}} s_n(S)^p\right)^{1/p}$. Since these norms are defined in terms of $\ell^p$-norms of sequences, we get that $\|\cdot\|_{\bo}\leq \|\cdot\|_p\leq \|\cdot\|_1$ for $1\leq p \leq \infty$. Furthermore, the spaces $\SC^p$ are ideals in $\bo$, meaning that $A\in \bo$ and $T\in \SC^p$ implies that $AT,TA \in \SC^p$  \cite[Thm. 2.7]{Simon:2010wc}. 

\subsubsection{The trace and trace class operators}
Recall that an operator $S\in \bo$ is \textit{positive} if $\inner{S\psi}{\psi}\geq 0$ for any $\psi \in \HS$. For a positive operator $S\in \bo$, the \textit{trace} of $S$ is defined to be
 \begin{equation} \label{eq:defoftrace}
  \tr(S)=\sum_{n\in \N} \inner{Se_n}{e_n},
\end{equation}
where $\{e_n\}_{n\in \N}$ is an orthonormal basis for $\HS$. This definition is independent of the orthonormal basis used, and the trace is linear and satisfies $\tr(ST)=\tr(TS)$ \cite{Reed:1980}. However, the expression in \eqref{eq:defoftrace} may well be infinite, and is not well-defined for a general non-positive operator $S$. If $S\in \tco$, then $\tr(S)$ is well-defined and a simple calculation shows that 
\begin{equation*}
  \tr(S)=\sum_{n\in \N} s_n(S),
\end{equation*}
where the sum of singular values converges by the definition of $\tco$. For this reason the class $\tco$ is often referred to as \textit{trace class operators}. By a celebrated theorem due to Lidskii, the trace $\tr(S)$ of $S\in \tco$ equals the sum $\sum_{i=1}^{\infty}\lambda_i$ of the eigenvalues $\{\lambda_i\}_{i\in \N}$ of $S$, where the eigenvalues are counted with algebraic multiplicity \cite{Simon:2010wc}. 

Using the trace we may state the duality relations of the Schatten $p$-classes \cite[Thm 2.8 and 3.2]{Simon:2010wc}. 
\begin{lem}
	Let $1\leq p<\infty$, and let $q$ be the number determined by $\frac{1}{p}+\frac{1}{q}=1$. The dual space of $\SC^p$ is $\SC^{q}$, and the duality may be given by 
		\begin{equation*}
  		\inner{T}{S}_{\SC^q,\SC^p}=\tr(TS) 
		\end{equation*}
		for $S \in \SC^p$ and $T\in \SC^{q}$.
\end{lem} 
Another well-known Schatten class is $\SC^2$, known as the \textit{Hilbert-Schmidt operators}. $\SC^2$ is a Hilbert space under the inner product $\inner{S}{T}_{\SC^2}:=\tr(ST^*)$ for $S,T\in \SC^2$.
\begin{rem}
	The Schatten classes behave analogously to the $L^p$-spaces of functions -- the duality relations are the same, and both $L^1(\Rdd)$ and $\tco$ have a natural linear functional given by the integral and trace, respectively. The intuition that $L^p$ corresponds to $\SC^p$ will often be useful, and is strengthened by the convolutions defined in section \ref{sec:werner}.
\end{rem}

\subsubsection{Vector-valued integration} \label{sec:weakint}
We will need to integrate operator-valued functions $G:\Rdd \to \bo$ of the form $G(z)=g(z)F(z)$, where $g\in L^1(\Rdd)$ and $F:\Rdd\to \bo$ is measurable, bounded and strongly continuous. The operator-valued integral $\iint_{\Rdd} g(z)F(z) \ dz \in \bo$ is defined in a weak and pointwise sense: for any $\psi \in \HS$ we define $\left(\iint_{\Rdd} g(z)F(z) \ dz\right)\psi$ by
\begin{equation*}
  \inner{\left(\iint_{\Rdd} g(z)F(z) \ dz \right) \psi}{\phi}=\iint_{\Rdd} g(z)\inner{F(z)\psi}{\phi} \ dz
\end{equation*}
   for any $\phi \in \HS$. This defines an operator $\iint_{\Rdd} g(z)F(z) \ dz$, and we get the norm estimate $\|\iint_{\Rdd} g(z) F(z) \ dz\|_{\bo}\leq \|g\|_{L^1} \sup_{z\in \Rdd} \|F(z)\|_{\bo}$ \cite{Luef:2017vs}.

\subsection{Localization operators, STFT-filters and multiwindow STFT-filters} \label{sec:filters}
Given a function $f$ on $\R^{2d}$ and $\varphi_1, \varphi_2 \in L^2(\R^{d})$ , we define the \textit{localization operator} (or STFT-filter \cite{Kozek:1997}) $\mathcal{A}_f^{\varphi_1,\varphi_2}$ \textit{with mask $f$ and windows $\varphi_1,\varphi_2$} by
\begin{equation*}
  \mathcal{A}_f^{\varphi_1, \varphi_2}\psi = \iint_{\R^{2d}}f(z) V_{\varphi_1}\psi(z) \pi(z) \varphi_2 \ dz
\end{equation*}
for $\psi\in L^2(\R^d)$,  where the integral is interpreted in the weak sense discussed above. We will in particular be interested in the case where $\varphi_1=\varphi_2$ and $f=\chi_\Omega$ is the characteristic function of some measurable subset $\Omega\subset \R^{2d}$, and we will write $\mathcal{A}_\Omega^\varphi:=\mathcal{A}^{\varphi,\varphi}_{\chi_\Omega}$ in this case. 

We will follow Kozek \cite{Kozek:1997} and call any operator $H$ of the form
\begin{equation*}\label{eq:mw}
 H= \sum_{n} \lambda_n \mathcal{A}_f^{\varphi_{n,1},\varphi_{n,2}}
\end{equation*}
a \textit{multiwindow STFT-filter}, where $\{ \lambda_n \}_{n\in \N}$ is a sequence of complex numbers and $\{\varphi_{n,1}\}_{n\in \N}$ and $\{\varphi_{n,2}\}_{n\in \N}$ are sequences of functions in $L^2(\R^d)$. Hence a multiwindow STFT-filter is a possibly infinite linear combination of localization operators with common mask $f$. We will return to the question of convergence of the sum in equation \eqref{eq:mw} in section \ref{sec:rigour}. For further information on filters and their use in the engineering literature the reader may consult, for instance, \cite{Hlawatsch:1992, Kozek:1997, Hlawatsch:1994,Matz:2002}.

\subsection{Convolutions of operators and functions} \label{sec:werner}
This section introduces the theory of convolutions of operators and functions due to Werner \cite{Werner:1984}. In order to introduce these convolution operations, we will first need to define a shift for operators. For $z\in \R^{2d}$ and $A\in \bo$, we define the operator $\alpha_z(A)$ by
\begin{equation*}
  \alpha_z(A)=\pi(z)A\pi(z)^*.
\end{equation*}
It is easily confirmed that $\alpha_z\alpha_{z'}=\alpha_{z+z'}$, and we will informally think of $\alpha$ as a shift or translation of operators. The interpretation of $\alpha$ as a shift of operators has also been remarked in the signal processing literature by Kozek \cite{Kozek:1992,Kozek:1997}.

Similarly we define the analogue of the involution $f\mapsto \check{f}$ of a function, for an operator $A\in \bo$ by
\begin{equation*}
  \check{A}=PAP,
\end{equation*}
where $P$ is the parity operator $P\psi(x)=\psi(-x)$ for $\psi \in L^2(\R^d)$. 
The intuition that $\alpha$ is a shift of operators is supported by considering the Weyl symbol \cite{Luef:2017vs,Kozek:1992}.
\begin{lem} \label{lem:shiftofweyl}
Let $f\in L^1(\R^{2d})$, and let $L_f$ be the Weyl transform of $f$. 
\begin{itemize}
	\item $\alpha_z(L_f)=L_{T_zf}$ for $z\in \R^{2d}$.
	\item $\check{L_f}=L_{\check{f}}$.
\end{itemize}	
	
\end{lem}

Using $\alpha$, Werner defined a convolution operation between functions and operators \cite{Werner:1984}. If $f \in L^1(\R^{2d})$ and $S \in \tco$ we define the \textit{operator} $f\star S$ by
\begin{equation*}
  f\star S := S\star f := \iint_{\R^{2d}}f(y)\alpha_y(S) \ dy
\end{equation*}
where the integral is interpreted as in section \ref{sec:weakint}. Then $f\star S \in \tco$ and $\|f\star S \|_{\tco}\leq \|f \|_{L^1}\|S\|_{\tco}$ \cite[Prop. 2.5]{Luef:2017vs}.

For two operators $S,T \in \tco$, Werner defined the \textit{function} $S\star T$ by 
\begin{equation*}
  S \star T(z) = \tr(S\alpha_z (\check{T})) 
\end{equation*}
for $z\in \R^{2d}$. 

\begin{rem}
	The notation $\star$ may therefore denote either the convolution of two functions or the convolution of an operator with a function. The correct interpretation will be clear from the context.
\end{rem}
The following result shows that $S\star T\in L^1(\R^{2d})$ for $S,T\in \tco$ and provides an important formula for its integral \cite[Lem. 3.1]{Werner:1984}. In the simplest case where $S$ and $T$ are rank-one operators, this formula is the so-called Moyal identity for the STFT\cite[p. 57]{Folland:1989}.
\begin{lem} 
\label{lem:werner}
Let $S,T \in \tco$. The function $z \mapsto \tr(S\alpha_zT)$ for $z\in \R^{2d}$ is integrable and $\|\tr(S\alpha_zT)\|_{L^1} \leq \|S\|_{\tco} \|T\|_{\tco}$.

Furthermore,
\begin{equation*}
	\iint_{\R^{2d}} \tr(S\alpha_zT) \ dz = \tr(S)\tr(T).
\end{equation*}
\end{lem}

The convolutions can be defined on other $L^p$-spaces and Schatten $p$-classes by duality \cite{Luef:2017vs,Werner:1984}. As an important example, the convolution $f\star S\in \bo$ for $f\in L^{\infty}(\Rdd)$ and $S\in \tco$ is defined by the relation
\begin{align} \label{eq:convolutionsofbounded}
  &\inner{f\star S}{T}_{\bo,\tco}=\inner{f}{\check{S}\star T}_{L^{\infty},L^1} && \text{for any $T\in \tco$},
\end{align}
 By writing these dualities explicitly, the definition becomes
 \begin{align} \label{eq:convolutionsofboundedexplicit}
  &\tr((f\star S)T)=\iint_{\Rdd} f(z) (\check{S}\star T)(z) \ dz && \text{for any $T\in \tco$}.
\end{align}
 
  When extended to other functions and operator spaces, the convolutions satisfy a version of Young's inequality.

\begin{prop} \label{prop:convschatten}
	Let $1\leq p,q,r \leq \infty$ be such that $\frac{1}{p}+\frac{1}{q}=1+\frac{1}{r}$. If $f\in L^p(\R^{2d}), S \in \SC^p$ and $T\in \SC^q$, then the following convolutions may be defined and satisfy the norm estimates
	\begin{align*}
  		\|f\star T\|_{\SC^r} &\leq \|f\|_{L^p} \|T\|_{\SC^q}, \\
  		\|S\star T\|_{L^r} &\leq \|S\|_{\SC^p} \|T\|_{\SC^q}.  		
	\end{align*}
	\end{prop}

The convolutions of operators and functions are associative, a fact that is non-trivial since the convolutions between operators and functions can produce both operators and functions as output \cite{Luef:2017vs, Werner:1984}. Commutativity and bilinearity, however, follows straight from the definitions.

Furthermore, the convolutions preserve positivity and identity elements \cite{Skrettingland:2017}.
\begin{prop} \label{prop:positiveandidentity}
	\begin{enumerate}
		\item If $S,T \in \bo$ are positive operators and $f$ is a positive function, then $f\star S$ is a positive operator and $S\star T$ is a positive function.
		\item If $1$ is the constant function $1(z)=1$ for $z\in \Rdd$ and $I$ is the identity operator on $\HS$, then $1\star S=I$ and $I\star S=1$ for every $S\in \tco$. 
	\end{enumerate}
\end{prop}

The convolutions make the Schatten classes $\SC^p$ into \textit{Banach modules} over $L^1(\Rdd)$ if the module multiplication is defined by $(f,S)\mapsto f\star S$ for $f\in L^1(\Rdd)$ and $S\in \SC^p$, \cite{Luef:2017vs}. By using the Cohen-Hewitt theorem for Banach modules \cite{Graven:1974}, one obtains that any operator in $\SC^p$ for $p<\infty$ can be written as a convolution\cite[Prop. 7.4]{Luef:2017vs}.
\begin{prop} \label{prop:banachmod}
	Given $T\in \SC^p$ for $p<\infty$, there exists $f\in L^1(\R^{2d})$ and $S\in \SC^p$ such that $T=f\star S$.
\end{prop}

\subsection{Localization operators and spectrograms as convolutions} 
In \cite{Luef:2017vs} we established that Werner's convolutions provide a conceptual framework for localization operators, as shown by the following result.

\begin{lem} \label{lem:locasconv}
	Let $f$ be a function on $\R^{2d}$ and $\varphi_1, \varphi_2 \in L^2(\R^d)$ with $\|\varphi_1\|_{L^2}=\|\varphi_2\|_{L^2}=1$. Then the localization operator $\mathcal{A}_f^{\varphi_1,\varphi_2}$ can be expressed as the convolution of the function $f$ and the rank-one operator $\varphi_2\otimes \varphi_1$,
	\begin{equation*}
  \mathcal{A}_f^{\varphi_1,\varphi_2} = f\star (\varphi_2 \otimes \varphi_1).
\end{equation*}
\end{lem}
Similarly, the convolution of two rank-one operators reduces to a familiar object in the simplest case -- namely the spectrogram.
\begin{lem} \label{lem:spectrogramasconvolution}
		Let $\phi,\psi,\xi,\eta \in \HS$. Then the function $V_{\eta}\phi(z) \overline{V_{\xi}\psi(z)}$ may be expressed as the convolution of two rank-one operators,
\begin{equation*}
(\phi\otimes \psi)\star (\check{\xi} \otimes \check{\eta})(z)= V_{\eta}\phi(z) \overline{V_{\xi}\psi(z)}. 
\end{equation*}
for $z\in \R^{2d}$. In particular, if $\eta=\psi$ and $\psi=\phi$, then $(\phi\otimes \phi)\star (\check{\eta} \otimes \check{\eta})$ is the \textit{spectrogram} $|V_{\eta}\phi|^2$.

\end{lem}
Note that in the physics literature the spectrogram $|V_{\eta}\phi|^2$ is called the \textit{Husimi function} of $\phi$ when $\eta$ is a Gaussian \cite{Keller:2017}.
\subsection{The Fourier-Wigner transform of operators} \label{sec:fourier}
 For operators $S\in \tco$, the \emph{Fourier-Wigner transform} $\F_WS$ of $S$ is the function given by
	\begin{equation*}
	\F_W S(z)=e^{-\pi i x \cdot \omega}\tr(\pi(-z)S)
	\end{equation*}
	for $z\in \R^{2d}$. 
	In the special case of an operator of rank one, the Fourier-Wigner transform is the ambiguity function \cite[Lemma 6.1]{Luef:2017vs}.
\begin{lem} \label{lem:FWrankone}
If $\varphi_1,\varphi_2 \in \HS$, then
	$\F_W(\varphi_2\otimes \varphi_1)(z)=A(\varphi_2,\varphi_1)(z).$
\end{lem}
 The Fourier-Wigner transform has many properties analogous to those of the Fourier transform of functions\cite{Werner:1984, Luef:2017vs}. It extends to a unitary operator $\F_W: \SC^2\to L^2(\R^{2d})$, and by the following proposition it interacts with the convolutions defined by Werner in the expected way.  

\begin{prop}
	\label{prop:convolutionandft}
	Let $f \in L^1(\R^{2d})$ and $S,T \in \tco$. 
	\begin{enumerate}
		\item $\F_{\sigma}(S \star T)=\F_W(S)\F_W(T)$.
		\item $\F_W(f \star S)=\F_{\sigma}(f)\F_W(S)$.  
	\end{enumerate}	
\end{prop} 

In time-frequency analysis and signal processing, operators are sometimes studied by considering the so-called \textit{spreading function} \cite{Feichtinger:1998}, which expresses the operator as an infinite linear combination of time-frequency shifts. In fact, the Fourier-Wigner transform and the spreading function differ only by a phase factor \cite{Luef:2017vs}.
\begin{prop} \label{prop:fwspreading}
\begin{enumerate}
	\item If $S\in \tco$ has spreading function $f\in L^1(\R^{2d})$, i.e.
	\begin{equation*}
  S=\int_{\R^{2d}} f(z) \pi(z) \ dz,
\end{equation*}
where the integral is interpreted as in section \ref{sec:weakint}, then 
\begin{equation*}
  \F_W(S)(z)= e^{i\pi x \cdot \omega}f(z).
\end{equation*}
\item The Weyl symbol $a_S$ of $S\in \tco$ is given by
\begin{equation*}
  a_S=\F_{\sigma}\F_W(S).
\end{equation*}
\end{enumerate}
\end{prop}

As for the Fourier transform of functions, there is also a Hausdorff-Young inequality associated to the Fourier-Wigner transform \cite{Werner:1984, Luef:2017vs}.
\begin{prop} \label{prop:hausyoung}
	Let $1\leq p \leq 2$ and let $q$ be the conjugate exponent determined by $\frac{1}{p}+\frac{1}{q}=1$. If $S\in \SC^p$, then $\F_W(S)\in L^q(\R^{2d})$  with norm estimate
  	\begin{equation*}
  \|\F_W(S)\|_{L^q} \leq \|S\|_{\SC^p}.
\end{equation*}
\end{prop} 
Using Lieb's uncertainty principle \cite{Grochenig:2001,li90-1} we can improve this result in a special case \cite{Luef:2017vs}.
 \begin{cor} \label{cor:generalizedlieb}
 Let $2\leq p <\infty$. If $S\in \tco$, then 
 \begin{equation*}
  \|\F_W(S)\|_{L^p}\leq  \left(\frac{2}{p}\right)^{d/p}
 \|S\|_{\tco}.
\end{equation*}
	
 \end{cor}
 
 \subsubsection{Tauberian theorems for operators}
 Werner \cite{Werner:1984} has proved a version of Wiener's Tauberian theorem for operators. The theorem was later generalized in \cite{Kiukas:2012gt}, and more equivalent statements and a proof may be found in \cite{Luef:2017vs,Kiukas:2012gt}. We state the relevant parts of the theorem for our purposes.

\begin{thm} \label{thm:wiener}
Let $S\in \tco$. 
	\begin{itemize}
\item [(a)] The following are equivalent.  
\begin{enumerate}[label=(a\arabic*)]
		\item The set $\{z\in \R^{2d}:\F_W(S)=0\}$ is empty.
		\item If $f \in L^{\infty}(\mathbb{R}^{2d})$ and $f\star S=0$, then $f=0$.
		\item $ L^1(\R^{2d})\star S$ is dense in $\SC^1$.
		\item If $T\in \bo$ and $S\star T=0$, then $T=0$.
	\end{enumerate}
\item [(b)] The following are equivalent.  
\begin{enumerate}[label=(b\arabic*)]
		\item The set $\{z\in \R^{2d}:\F_W(S)=0\}$ has Lebesgue measure $0$.
		\item If $f \in L^2(\mathbb{R}^{2d})$ and $f\star S=0$, then $f=0$.
		\item $ L^2(\R^{2d})\star S$ is dense in $\SC^2$.
		\item If $T\in \SC^2$ and $S\star T=0$, then $T=0$.
	\end{enumerate}
\item [(c)] The following are equivalent.  
\begin{enumerate}[label=(c\arabic*)]
		\item The set $\{z\in \R^{2d}:\F_W(S)=0\}$ has dense complement.
		\item If $f \in L^1(\mathbb{R}^{2d})$ and $f\star S=0$, then $f=0$.
		\item $ L^{\infty}(\R^{2d})\star S$ is weak* dense in $\bo$.
		\item If $T\in \tco$ and $S\star T=0$, then $T=0$.
	\end{enumerate}
\end{itemize}
\end{thm}

 \subsection{Schwartz operators and tempered distributions}
 The theory of convolutions and Fourier transforms of operators can be extended to more general objects than bounded operators, just as the convolution and Fourier transform of functions is extended from the $L^p$-spaces to tempered distributions. To define this extension, we start by defining two classes of operators. We let $\schwartz$ be the set of pseudodifferential operators with Weyl symbol in the Schwartz class $\mathcal{S}(\R^{2d})$, and we let $\tempdist$ be the set of pseudodifferential operators with Weyl symbol in the tempered distributions $\mathcal{S}^{\prime}(\R^{2d})$. These sets of operators were studied in detail by Keyl et al. in \cite{Keyl:2016}. They show that $\schwartz$ may be equipped with a topology making it a Frechet space, and that $\tempdist$ is the topological dual space of $\schwartz$ in this topology. Hence one may define convolutions and Fourier transforms on $\tempdist$ using duality. We summarize the main results in the following proposition, and refer to section 5 of \cite{Keyl:2016} for proofs.
 
 \begin{prop} \label{prop:schwartzoperators}
 	
 	\begin{enumerate}
 		\item Let $S,T\in \schwartz$, $A\in \tempdist$,$f\in \mathcal{S}(\R^{2d})$ and $\phi\in \mathcal{S}^{\prime}(\R^{2d})$. The following convolutions may be defined:
 		\begin{align*}
  S\star T &\in \mathcal{S}(\R^{2d})   &&f\star S \in \schwartz \\
  S\star A &\in \mathcal{S}^{\prime}(\R^{2d})  && \phi \star S \in \tempdist \\
  & && f\star A \in \tempdist.
  \end{align*}
  \item The definitions in part (1) are compatible with those in section \ref{sec:werner} whenever both are applicable.
  \item The Fourier-Wigner transform may be extended to a topological isomorphism $\F_W:\tempdist\to \mathcal{S}^{\prime}(\R^{2d})$.
  \item The relations  $\F_{\sigma} (S \star T)=\F_W(S)\F_W(T)$ and $\F_W(f \star S)=\F_{\sigma}(f)\F_W(S)$ still hold for operators $S$,$T$ and a function $f$ whenever the convolutions are defined by part (1).
  \item The Weyl symbol of $A\in \tempdist$ is given by $\F_{\sigma}\F_W(A)$. 
 	\end{enumerate}
 \end{prop} 
 \begin{rem}
 	By the Schwartz kernel theorem (see \cite{Hormander:1983}), we know that we may identify $\tempdist$ with the continuous operators from $\mathcal{S}(\Rdd)$ to $\mathcal{S}^{\prime}(\Rdd)$.
 \end{rem}
 
 \subsection{Positive operator valued measures} \label{sec:povm}
In section \ref{sec:genlocandpovm} of this paper we will argue that the notion of a positive operator valued measure is a natural framework for localization operators and Cohen's class of time-frequency distributions. This notion is more commonly used in operator theory and quantum mechanics  \cite{Beneduci:2014}. We recall the basic concepts.
\begin{defn} \label{def:povm}
	Let $\mathcal{B}(\R^{2d})$ denote the $\sigma$-algebra of Borel subsets of $\R^{2d}$.  A positive operator valued measure (POVM) on $\R^{2d}$ is a mapping $F:\mathcal{B}(\R^{2d})\to B(\HS)$ such that
	\begin{enumerate}
		\item $F(M)$ is a positive operator for any $M\in \mathcal{B}(\R^{2d})$,
		\item $F(\R^{2d})$ is the identity operator on $\HS$,
		\item $F\left( \cup_{i\in \mathbb{N}} M_i \right)=\sum_{i\in \mathbb{N}} F(M_i)$ for any countable collection of disjoint, measurable subsets $\{M_i\}_{i\in \mathbb{N}}$ of $\R^{2d}$, where the sum converges in the weak operator topology.
	\end{enumerate}
\end{defn}
Hence a POVM on $\R^{2d}$ assigns a positive operator on $\HS$ to each Borel subset of $\R^{2d}$. Convergence in the weak operator topology of the sum $\sum_{i\in \mathbb{N}} F(M_i)$ to the operator $T:=F\left( \cup_{i\in \mathbb{N}} M_i\right)$ means that for any $\psi,\phi\in \HS$ we have $\sum_{i\in \mathbb{N}} \inner{F(M_i)\psi}{\phi}=\inner{T\psi}{\phi}$.
Any POVM $F$ appearing in this text will be \textit{covariant}, meaning that $ \alpha_z(F(M))=F(M+z) $ for any $z\in \R^{2d}$ and $M\in \mathcal{B}(\R^{2d})$, where $M+z=\{m+z:m\in M\}$ and $\alpha$ is the shift of operators defined in section \ref{sec:werner}. 

\subsubsection{Integration and the probability measures associated to a POVM}
Let $F$ be a fixed POVM. For each $\psi\in \HS$ with $\|\psi\|_{L^2}=1$, $F$ allows us to construct a probability measure $\mu_{\psi}^F$ on $\R^{2d}$ by defining
\begin{equation*}
    \mu_{\psi}^F(\Omega) = \inner{F(\Omega)\psi}{\psi}
\end{equation*}
for $\Omega \subset \R^{2d}$. 

Using the measures $\mu_{\psi}^F$, we may define a notion of integration w.r.t. the POVM $F$ \cite[Sec. 5 Thm. 9]{Berberian:1966}
\begin{lem} \label{lem:povmintegration}
	If $f:\Rdd \to \C$ is a measurable, bounded function, then there exists a unique operator $A_f\in \bo$ such that $\inner{A_f\psi}{\psi}=\iint_{\Rdd} f(z) d\mu_{\psi}^F$ for any $\psi\in \HS$.
\end{lem}
  We denote the operator $A_f$ by $\iint_{\Rdd}f(z) dF$. For $f\in L^{\infty}(\Rdd)$ and $\Omega \subset \Rdd$, we define $\iint_{\Omega}f \ dF:=\iint_{\Rdd} \chi_{\Omega}f \ dF$. It is easily seen that $\iint_{\Omega}\ dF=F(\Omega)$.

\section{The time-frequency concentration of the spreading function}
When considering a filter $H$, it is often of interest to determine the amount of "spreading" in time and frequency that $H$ performs on a signal. By proposition \ref{prop:fwspreading}, the Fourier-Wigner function $\F_W(H)$ is, up to a phase factor, the spreading function of $H$. Hence the Fourier-Wigner transform $\F_W(H)(z)$ is the weight of the time-frequency shift $\pi(z)$ when $H$ is decomposed as a linear combination of time-frequency shifts:
\begin{equation*}
  H=\iint_{\Rdd} e^{i\pi x \cdot \omega} \F_W(H)(z) \pi(z) \ dz,
\end{equation*}
where the integral is interpreted in the sense of section \ref{sec:weakint} for $H\in \tco$. 
For instance, an operator which only shifts signals slightly in time and frequency will have a spreading function concentrated around $0$ in $\R^{2d}$. \\ 

To measure the effect of $H$ on a signal, we would therefore like to obtain bounds on the concentration of the spreading function, or equivalently of $\F_W(H)$.  In fact, the Hausdorff Young inequality in proposition \ref{prop:hausyoung} does exactly this. By this inequality, if $1\leq p \leq 2$ and $\frac{1}{p}+\frac{1}{q}=1$, then if $H\in \SC^p$ we get

  	\begin{equation} \label{eq:hausyoung}
 \left(\iint_{\R^{2d}} |\F_W(H)|^q \ dz \right)^{1/q} \leq \|H\|_{\SC^p}.
\end{equation}
Hence we can interpret the Hausdorff Young inequality as saying that the Schatten class norm of $H$ provides information on the concentration of the spreading function of $H$. If $H$ is trace class, then the above inequality holds for all $2\leq q <\infty$, and we may replace $\|H\|_{\SC^p}$ by $\|H\|_{\tco}$, since $\|H\|_{\SC^p}\leq \|H\|_{\tco}$ for any $p\geq 1$. 

\begin{rem}
 Since the Fourier-Wigner transform is unitary from $\SC^2$ to $L^2(\R^{2d})$ \cite{Luef:2017vs}, we actually have an equality in equation \eqref{eq:hausyoung} for $p=q=2$.
\end{rem}

Following the reasoning used by Gr\"ochenig to prove an uncertainty principle for functions in \cite[Thm. 3.3.3.]{Grochenig:2001}, we can use corollary \ref{cor:generalizedlieb} to obtain an uncertainty principle for spreading functions of filters.
\begin{thm}
Let $S\in \tco$ and let $\Omega\subset \R^{2d}$ with $\mu(\Omega)<\infty$ 	and assume that
	\begin{equation*}
  \iint_{\Omega} |\F_W(S)(z)|^2 \ dz \geq 1-\epsilon
\end{equation*}
for some $\epsilon \geq 0$. For any $p>2$ we then have
\begin{equation*}
  \mu(\Omega)\geq \frac{(1-\epsilon)^{p/(p-2)} \left(\frac{p}{2}\right)^{2d/(p-2)}}{\|S\|_{\tco}^{2p/(p-2)}}.
\end{equation*}
In particular, for $p=4$ we obtain
\begin{equation*}
  \mu(\Omega)\geq \frac{(1-\epsilon)^{2} 2^d}{\|S\|^4_{\tco}}.
\end{equation*}	
\end{thm}
\begin{proof}
	By H\"older's inequality with $p'=\frac{p}{2}$ and $q'=\frac{p}{p-2}$, we find that
\begin{align*}
  1-\epsilon&\leq  \iint_{\Omega} |\F_W(S)(z)|^2 \ dz\\
  &\leq \left(\iint_{\R^{2d}} |\F_W(S)|^{2\frac{p}{2}} \ dz \right)^{2/p} \left(\iint_{\R^{2d}} \chi_{\Omega}(z)^{\frac{p}{p-2}} \ dz\right)^{(p-2)/p} \\
  &\leq \left(\frac{2}{p}\right)^{2d/p}\|S\|_{\tco}^2 \mu(\Omega)^{(p-2)/p},
\end{align*}
where the last inequality follows from corollary \ref{cor:generalizedlieb}. Rearranging this inequality, we obtain 
\begin{equation*}
  \mu(\Omega)\geq \frac{(1-\epsilon)^{p/(p-2)} \left(\frac{p}{2}\right)^{2d/(p-2)}}{\|S\|_{\tco}^{2p/(p-2)}}.
\end{equation*}

\end{proof}
One interpretation of this uncertainty principle is that a well-concentrated spreading function comes at the cost of a large trace class norm. As an example, consider the special case of an \textit{underspread} trace class operator $S$, meaning that the support of $S$ is contained in some bounded subset $\Omega\subset \R^{2d}$ with $\mu(\Omega)<<1$ \cite{Kozek:1997us}. Assume that $S$ is normalized in the sense that $\|S\|_{\SC^2}=\iint_{\R^{2d}} |\F_W(S)|^2 \ dz =1$. By assumption we then have
\begin{equation*}
  \iint_{\Omega} |\F_W(S)|^2 \ dz = 1,
\end{equation*}
and by the previous result with $\epsilon=0$ we conclude that 
\begin{equation*}
   1>>\mu(\Omega)\geq \frac{2^d}{\|S\|^4_{\tco}},
\end{equation*}
hence $\|S\|_{\tco}>>1$.

\section{Multiwindow STFT-filters are convolutions} \label{sec:rigour}
One aim of the recent paper \cite{Luef:2017vs} was to apply Werner's theory of convolutions to localization operators (or STFT-filters \cite{Feichtinger:2001,Kozek:1997}) using the identity
\begin{equation*}
    \mathcal{A}_f^{\varphi_1,\varphi_2} = f\star (\varphi_2 \otimes \varphi_1).
\end{equation*}
There are several advantages to this approach. Proposition \ref{prop:convschatten} provides a simple relationship between the properties of the mask $f$ and the operator  $\mathcal{A}_f^{\varphi_1,\varphi_2}$, the Fourier-Wigner transform is a useful tool for considering the Weyl symbol of $\mathcal{A}_f^{\varphi_1,\varphi_2}$ and the Tauberian theorem (theorem \ref{thm:wiener}) is a powerful tool to deduce new insights into localization operators. We will now show that multiwindow STFT-filters also allow a description in terms of convolutions.

In section \ref{sec:filters}, a multiwindow STFT-filter $H$ was defined as a linear combination of localization operators with a fixed mask $f$:
\begin{equation*}
    H=\sum_{n=1}^N \lambda_n \mathcal{A}_{f}^{\varphi_{n,1},\varphi_{n,2}}
\end{equation*}
for some sequence $\{\lambda_n\}_{n\in \N}$ in $\C$ and sequences $\{\varphi_{n,1}\}_{n\in \N}$ and $\{\varphi_{n,2}\}_{n\in \N}$ in $L^2(\R^{d})$.
Since any $\mathcal{A}_{f}^{\varphi_{n,1},\varphi_{n,2}}$ may be written as the convolution $f\star (\varphi_{n,2}\otimes \varphi_{n,1})$, we get by the linearity of convolutions that 
\begin{equation*}
  H=f\star \sum_{n=1}^N \lambda_n \varphi_{n,2}\otimes \varphi_{n,1}.
\end{equation*}
Hence $H$ is the convolution of $f$ with the operator $\sum_{n=1}^N \lambda_n \varphi_{n,2}\otimes \varphi_{n,1}$. When $N$ is finite, the sum $\sum_{n=1}^N \lambda_n \varphi_{n,2}\otimes \varphi_{n,1}$ is always a trace class operator, so by proposition \ref{prop:convschatten} we may pick the mask $f\in L^p(\R^{2d})$ for any $1\leq p \leq \infty$ and obtain a bounded operator. However, if follow Hlawatsch and Kozek \cite{Hlawatsch:1992} and introduce infinite linear combinations of localization operators, both the properties of the mask $f$ and convergence must be considered more carefully. 

\begin{prop} \label{prop:rigorous}
	Fix $1\leq p,q,r \leq \infty$ such that $\frac{1}{p}+\frac{1}{q}=1+\frac{1}{r}$. Let $\{\varphi_{n,1}\}_{n\in \N}$ and $\{\varphi_{n,2}\}_{n\in \N}$ be two orthonormal sequences in $L^2(\R^{d})$.  
	\begin{enumerate}
		\item If $\{\lambda_n\}_{n\in \N}\in \ell^p$ and $f\in L^q(\R^{2d})$, then the sum defining the multiwindow STFT-filter $ \sum_{n=1}^{\infty} \lambda_n \mathcal{A}_f^{\varphi_{n,1},\varphi_{n,2}}$ converges in $\SC^r$. Furthermore,
		\begin{equation*}
  \sum_{n=1}^{\infty} \lambda_n \mathcal{A}_f^{\varphi_{n,1},\varphi_{n,2}}=f\star \sum_{n=1}^{\infty} \lambda_n \varphi_{2,n}\otimes \varphi_{1,n}.
\end{equation*}
 
\item Conversely, any operator of the form $f\star S\in \SC^r$ for $f\in L^q(\R^{2d})$ and $S\in \SC^p$ can be written as a multiwindow STFT-filter with mask $f$. That is, there exists some sequence $\{\lambda_n\}_{n\in \N}\in \ell^p$ of positive numbers and $\{\varphi_{n,1}^{\prime}\}_{n\in \N}$, $\{\varphi_{n,2}^{\prime}\}_{n\in \N}$ two orthonormal sequences in $L^2(\R^{d})$ such that
\begin{equation*}
  f\star S=\sum_{n=1}^{\infty} \lambda_n \mathcal{A}_f^{\varphi^{\prime}_{n,1}, \varphi^{\prime}_{n,2}}
\end{equation*}
where the sum converges in $\SC^r$.
	\end{enumerate}
\end{prop}
\begin{proof}
	\begin{enumerate}
		\item The sum $\sum_{n=1}^{\infty} \lambda_n \varphi_{2,n}\otimes \varphi_{1,n}$ converges in the norm of $\SC^p$ to an operator in $\SC^p$ -- this follows from the definition of $\SC^p$ as those operators with singular values in $\ell^p$. By proposition \ref{prop:convschatten} the convolution $(h,S)\mapsto h\star S$ is continuous from $L^q(\R^{2d})\times \SC^p$ into $\SC^r$, and we may write
		\begin{align*}
  \sum_{n=1}^{\infty} \lambda_n \mathcal{A}_f^{\varphi_{n,1},\varphi_{n,2}}&= \sum_{n=1}^{\infty} \lambda_n f \star  (\varphi_{2,n}\otimes \varphi_{1,n}) \\
  &=f\star \sum_{n=1}^{\infty} \lambda_n   (\varphi_{2,n}\otimes \varphi_{1,n}),
\end{align*}
where continuity considerations were used in the last step.
\item $S$ has a singular value decomposition
\begin{equation*}
  S=\sum_{n=1}^{\infty} \lambda_n \varphi^{\prime}_{2,n}\otimes \varphi^{\prime}_{1,n}
\end{equation*}
converging in the norm of $\SC^p$, with $\{\lambda_n\}_{n\in \N}\in \ell^p$ and $\{\varphi_{n,1}^{\prime}\}_{n\in \N}$, $\{\varphi_{n,2}^{\prime}\}_{n\in \N}$ two orthonormal sequences in $L^2(\R^{d})$. By the continuity properties of the convolutions, we can write
\begin{align*}
  f\star S &= f\star \sum_{n=1}^{\infty} \lambda_n \varphi^{\prime}_{2,n}\otimes \varphi^{\prime}_{1,n} \\
  &= \sum_{n=1}^{\infty} \lambda_n f\star  (\varphi^{\prime}_{2,n}\otimes \varphi^{\prime}_{1,n}) \\
  &=\sum_{n=1}^{\infty} \lambda_n \mathcal{A}_f^{\varphi^{\prime}_{n,1}, \varphi^{\prime}_{n,2}}.
\end{align*}

	\end{enumerate}
\end{proof}
\begin{rem}
	The setting in \cite{Hlawatsch:1992} consisted of a square-summable sequence $\{\lambda_n\}_{n\in \mathbb{N}}\in \ell^2$ and a mask $f$ with unspecified properties. The above proposition makes the relationship between properties of $\{\lambda_n\}_{n\in \mathbb{N}}$, $f$ and the multi-window STFT-filter more transparent, showing how properties of $\{\lambda_n\}_{n\in \mathbb{N}}$ and $f$ are reflected in Schatten class properties of the multi-window STFT-filter. In particular the proposition gives conditions on $\{\lambda_n\}_{n\in \mathbb{N}}$ and $f$ to guarantee that the filter is a well-defined bounded operator, analogous to the conditions for the convolutions of two functions to be well-defined by Young's inequality. 
\end{rem}

\begin{rem}
\begin{enumerate}
	\item By proposition \ref{prop:banachmod} \textit{any} operator $H\in \SC^p$ for $1\leq p <\infty$ can be written in the form $H=f \star S$ for $f\in L^1(\R^{2d})$ and $S\in \SC^p$. With this in mind, the study of multiwindow STFT-filters is the study of the Schatten classes $\SC^p$ from a certain perspective.
	\item By proposition \ref{prop:schwartzoperators}, one might also define multiwindow STFT-filters $f\star S$ when $f\in \mathcal{S}(\Rdd)$ and $S\in \tempdist$, or when $f\in \mathcal{S}^{\prime}(\Rdd)$ and $S\in \schwartz$. 
\end{enumerate}
\end{rem}

\subsection{The Fourier-Wigner transform and multiwindow STFT-filters}

In \cite{Kozek:1992}, Kozek studied multiwindow STFT-filters by considering their Weyl symbols. One advantage from writing multiwindow STFT-filters using convolutions is that the relationship between such filters and their Weyl symbol becomes the relationship between convolutions and Fourier transforms. 
\begin{prop} \label{prop:weylsymbolofmultiwindow}
	Let $S\in \tco$ and $f\in \Ldd$. The Weyl symbol $a_{f\star S}$ of the multiwindow STFT $f\star S$ is given by
  $f\ast a_S$, where $a_S$ is the Weyl symbol of $S$.
\end{prop}
\begin{proof}
	By proposition \ref{prop:fwspreading}, $a_{f\star S}=\F_{\sigma}\F_W(f\star S)$. From proposition \ref{prop:convolutionandft} we know that $\F_W(f\star S)=\F_{\sigma}(f)\ast\F_W(S)$. Furthermore, we have the relation $\F_{\sigma}(gh)=\F_{\sigma}(g)\ast \F_{\sigma}(h)$ for $g,h \in \Ldd$; a fact that follows easily from the corresponding fact for the regular Fourier transform. Hence
	\begin{align*}
  a_{f\star S}=\F_{\sigma}\F_W(f\star S)&=\F_{\sigma}\left(\F_{\sigma}(f)\F_W(S)\right) \\
  &= f\ast \F_{\sigma}\F_W(S)=f\ast a_S,
\end{align*}
where we have used that $\F_{\sigma}$ is its own inverse.
\end{proof}
\begin{rem}
	Proposition 4.2 holds for more general $f$ and $S$, as long as the convolutions, $\F_{\sigma}$ and $\F_W$ are interpreted as their extensions to $\mathcal{S}^{\prime}(\Rdd)$ and $\tempdist$, respectively\cite{Keyl:2016}.
\end{rem}

Since the Weyl symbol of the operator $\varphi\otimes \varphi$ for $\varphi \in \HS$ is the Wigner function $W(\varphi,\varphi)$ \cite{Luef:2017vs}, we get in particular that the Weyl symbol $a_{\Omega}$ of a localization operator $\mathcal{A}_{\Omega}^\varphi =\chi_{\Omega}\star (\varphi\otimes \varphi)$ is 
\begin{equation*}
  a_{\Omega}=\chi_{\Omega}\ast W(\varphi,\varphi),
\end{equation*}
as is well-known \cite{Cordero:2003}.
\begin{rem}
	Consider $\varphi_1, \varphi_2 \in L^2(\R^d)$. By the same arguments as above we get that the Weyl symbol of the localization operator $f\star (\varphi_2 \otimes \varphi_1)$ is $f\ast W(\varphi_2,\varphi_1)$. When Kozek and Hlawatsch generalized from localization operators (or STFT-filters) $f\star (\varphi_2\otimes \varphi_1)$ to multiwindow STFT-filters $f\star S$ for $S\in \SC^2(\R^d)$ in \cite{Hlawatsch:1992}, they did so by considering the Weyl symbol $f\ast W(\varphi_2,\varphi_1)$ of a localization operator, and replaced $W(\varphi_2,\varphi_1)$ with an arbitrary function $k$ in $L^2(\R^{2d})$. Hence they considered the operator with Weyl symbol $f\ast k$, which by proposition \ref{prop:weylsymbolofmultiwindow} is the operator $f\star L_k$, where $L_k$ is the Weyl transform of $k$. Since  $\SC^2(\R^d)$ is exactly the set of bounded operators with Weyl symbol in $L^2(\R^{2d})$ \cite{Pool:1966}, the set of operators $f\star L_k$ for $k\in L^2(\R^{2d})$ equals the set of operators $f\star S$ for $S\in \SC^2(\Rd)$.
\end{rem}
\subsection{Density of multiwindow STFT-filters and uniqueness of masks} 
We will now fix an operator $S\in \tco$, and consider the corresponding set of multiwindow STFT-filters $f\star S$ for functions $f$. Using the Tauberian theorem for convolutions with operators (theorem \ref{thm:wiener}), we will be able to answer two questions about this set of filters. First, we ask whether any operator $T$ may be approximated by operators of the form $f\star S$, where $T$ belongs some specified Schatten $p$-class of operators. We then ask whether the mask $f$ is uniquely determined by the operator $f\star S$. 
\begin{prop} \label{prop:densityandreconstruction}
	Let $S\in \tco$. The following are equivalent.
	\begin{enumerate}
		\item The set $\{z\in \R^{2d}:\F_W(S)=0\}$ is empty.
		\item The set of multiwindow STFT-filters $ f\star S$ with $f\in L^1(\Rdd)$ is dense in $\SC^1$.
		\item Any mask $f \in L^{\infty}(\mathbb{R}^{2d})$ is uniquely determined by the multiwindow STFT-filter $f\star S$.
	\end{enumerate}
\end{prop}
\begin{proof}
	The result is simply a restatement of parts (a1), (a2) and (a3) of theorem \ref{thm:wiener} in the terminology of multiwindow STFT-filters.
\end{proof}
\begin{rem}
	Since the Weyl symbol of $S$ is $a_{S}=\F_{\sigma}\F_W(S)$, we see that $\F_W(S)=\F_{\sigma}a_S$. Hence part (1) of the result could equivalently have been formulated using the set of zeros of $\F_{\sigma}a_S$ -- the symplectic Fourier transform of the Weyl symbol of $S$.
\end{rem}
By relaxing the conditions on the set of zeros of the Fourier-Wigner transform of $S$, we obtain a result for Hilbert-Schmidt operators from theorem \ref{thm:wiener}.
\begin{prop} \label{prop:densityandreconstructionHS}
	Let $S\in \tco$. The following are equivalent.
	\begin{enumerate}
		\item The set $\{z\in \R^{2d}:\F_W(S)=0\}$ has Lebesgue measure zero.
		\item The set of multiwindow STFT-filters $ f\star S$ with $f\in L^2(\Rdd)$ is dense in $\SC^2$.
		\item Any mask $f \in L^2(\mathbb{R}^{2d})$ is uniquely determined by the multiwindow STFT-filter $f\star S$.
	\end{enumerate}
\end{prop}
With an even weaker assumption on the zeros of $\F_W(S)$, theorem \ref{thm:wiener} gives yet another result.
\begin{prop} \label{prop:densityandreconstruction2}
	Let $S\in \tco$. The following are equivalent.
	\begin{enumerate}
		\item The set $\{z\in \R^{2d}:\F_W(S)=0\}$ has dense complement in $\Rdd$.
		\item The set of multiwindow STFT-filters $ f\star S$ with $f\in L^{\infty}(\Rdd)$ is weak*-dense in $\bo$.
		\item Any mask $f \in L^{1}(\mathbb{R}^{2d})$ is uniquely determined by the multiwindow STFT-filter $f\star S$.
	\end{enumerate}
\end{prop}
If we pick $S=\varphi_2\otimes \varphi_1$ for $\varphi_1,\varphi_2 \in \HS$ in the three previous propositions, the conditions on the set of zeros of $\F_W(S)$ becomes a condition on the zeros of the ambiguity function $A(\varphi_2,\varphi_1)$. We noted this in \cite{Luef:2017vs}, where we generalized previous results from \cite{Bayer:2014td}. For such rank-one operators, proposition \ref{prop:densityandreconstruction} raises a natural question: Does there exist a pair of windows $\varphi_1,\varphi_2\in \HS$ such that $A(\varphi_2,\varphi_1)$ has no zeros, except when $\varphi_1=\varphi_2$ is a Gaussian? In the case where $\varphi_1=\varphi_2$ Hudson's theorem \cite{Grochenig:2001} requires $\varphi$ to be a Gaussian. Similarly, Toft \cite{Toft:2006} has shown that $V_{\varphi_1}\varphi_2$ can only be a positive function if $\varphi_1=\varphi_2$ is a Gaussian. However, the question of whether one may find $\varphi_1\neq\varphi_2$ such that $A(\varphi_2,\varphi_1)$ has no zeros remains open to the best of our knowledge.
 \begin{exmp} \label{exmp:reconstruction}
	Condition (1) of proposition \ref{prop:densityandreconstructionHS} is much weaker than the corresponding condition in proposition \ref{prop:densityandreconstruction}. It will for instance be satisfied by $S=h_n\otimes h_n$, where $h_n$ is the $n$'th Hermite function. In fact, $A(h_n,h_n)$ has a finite number of zeros, namely the zeros of some $n$'th Laguerre polynomials  \cite[p. 64]{Folland:1989}. 
\end{exmp}

\section{Mixed-state localization operators} \label{sec:generalizedlocalization}
Among the localization operators $\mathcal{A}_f^{\varphi_1,\varphi_2}$, those of the form $\mathcal{A}^\varphi_{\Omega}$ for some measurable $\Omega \subset \R^{2d}$ have a special interpretation: if $\psi\in L^2(\Rd)$, the signal $\mathcal{A}_{\Omega}^\varphi \psi$ is interpreted as the part of $\psi$ "living on" $\Omega$ \cite{Cordero:2003}, which explains the "localization" terminology. In section \ref{sec:rigour} we considered  multiwindow STFT-filters as a generalization of localization operators -- a natural question is then whether we can find some subset of the multiwindow STFT-filters where the "localization" interpretation above is still reasonable. We define a \textit{mixed-state localization operator} to be an operator $H_{\Omega}$ of the form
\begin{equation*}
  H_{\Omega}= \chi_{\Omega}\star S
\end{equation*}
where $\Omega \subset \Rdd$ is a measurable subset and $S$ is a positive trace class operator with $\tr(S)=1$. 

\begin{rem}
\begin{enumerate}
	\item The relationship between general localization operators $\mathcal{A}^{\varphi_1,\varphi_2}_f$ and those of the form $\mathcal{A}^{\varphi}_\Omega$ is the same as the relationship between multiwindow STFT-filters and mixed-state localization operators: A general localization operator may be written as $\mathcal{A}_{f}^{\varphi_1,\varphi_2}=f\star (\varphi_2\otimes \varphi_1)$, and the localization operators $\mathcal{A}_\Omega^\varphi$ are exactly those localization operators $f\star (\varphi_2\otimes \varphi_1)$ such that $f=\chi_{\Omega}$ for some $\Omega\subset \R^{2d}$ and $\varphi_2\otimes \varphi_1$ is a positive operator with $\tr(\varphi_2\otimes \varphi_1)=1$. This follows from the fact that $\varphi_2\otimes \varphi_1$ is positive if and only if $\varphi_1=\varphi_2$, and $\tr(\varphi_2\otimes \varphi_1)=\inner{\varphi_2}{\varphi_1}$.
	\item  In quantum mechanics, operators $\varphi \otimes \varphi$ with $\|\varphi\|_{L^2}=1$ describe \textit{pure states} of a system \cite{deGosson:2011wq}. More general states, the \textit{mixed states} are described by a positive operators $S\in \tco$ with $\tr(S)=1$. So a mixed-state localization operator is by definition given by the convolution of $\chi_\Omega$ with an operator describing a mixed state -- hence the name.
\end{enumerate}
\end{rem}

Given a mixed-state localization operator $\chi_{\Omega}\star S$, one might ask whether it is possible to recover information about the domain $\Omega$ from the operator $\chi_{\Omega}\star S$. The next proposition shows that the measure of $\Omega$ may be calculated from the eigenvalues $\chi_{\Omega}\star S$. In section \ref{sec:reconstructdomain} we will consider the problem of reconstructing the domain $\Omega$ in more detail.
\begin{prop} 
	Let $\Omega\subset \R^{2d}$ be a subset of finite Lebesgue measure, and let $S\in \tco$ be a positive operator with $\tr(S)=1$. Then 
	\begin{enumerate}
		\item $\tr(\chi_{\Omega}\star S)=\mu(\Omega)$, where $\mu$ is Lebesgue measure.
		\item If $\{\lambda_i\}_{i=1}^{N}$ are the eigenvalues of $\chi_{\Omega}\star S$ counted with algebraic multiplicity, then
		\begin{equation*}
  \sum_{i=1}^{N}\lambda_i=\mu(\Omega).
\end{equation*}
	\end{enumerate}
\end{prop}
\begin{proof}
	\begin{enumerate}
		\item By proposition \ref{prop:convolutionandft}, we have that $\F_W(\chi_{\Omega}\star S)(0)=\F_{\sigma}(\chi_{\Omega})(0)\F_W(S)(0)$, and by the definitions of $\F_W$ and $\F_{\sigma}$ we we have that $\F_{\sigma}(\chi_{\Omega})(0)\F_W(S)(0)=\mu(\Omega)\tr(S)=\mu(\Omega)$. 
		\item This follows from the first part along with Lidskii's equality from section \ref{sec:schatten}. 
	\end{enumerate}
\end{proof}
\begin{rem}
\begin{enumerate}
\item The proof of this proposition would work equally well if $\chi_{\Omega}$ is replaced by any $f\in L^1(\R^{2d})$, as long as $\mu(\Omega)$ is replaced by $\iint_{\R^{2d}} f(z) \ dz$.
	\item 	This result holds in particular for the localization operators $\mathcal{A}_{f}^{\varphi}$ by picking $S=\varphi\otimes \varphi$. In this context it is well-known, see for instance \cite{Feichtinger:2001}. The proposition therefore supports the intuition that $\chi_{\Omega}\star S$ is a generalized localization operator.
\end{enumerate}
\end{rem}

\subsection{A characterization of mixed-state localization operators}
By our definition of mixed-state localization operators, a positive trace class operator $S$ with $\tr(S)=1$ assigns to each domain $\Omega \subset \Rdd$ a mixed-state localization operator $\chi_{\Omega}\star S$. In fact, $f\star S$ belongs to $\bo$ for any $f\in L^{\infty}(\Rdd)$ by proposition \ref{prop:convschatten}, and in this way $S$ defines a bounded, linear mapping from $L^{\infty}(\Rdd)$ to $\bo$. The next theorem, originally due to Holevo \cite{Holevo:1979}, characterizes all bounded linear mappings $L^{\infty}(\Rdd)\to \bo$ of this form in terms of four properties. We provide an outline of the proof in \cite{Werner:1984} in our notation for completeness. The details may also be found in proposition 1 and lemma 3 in \cite{Kiukas:2006} where the result is proved in a more general setting.
\begin{thm} \label{thm:poscorrule}
	Let $\Gamma:L^{\infty}(\R^{2d})\to \bo$ be a linear operator satisfying
	\begin{enumerate}
		\item $\Gamma(\chi_{\R^{2d}})=I$, where $I$ the identity operator,
		\item $\Gamma(T_zf)=\alpha_z(\Gamma(f))$ for any $z\in \R^{2d}$ and $f\in L^{\infty}(\R^{2d})$,
		\item $\Gamma(f)$ is a positive operator whenever $f$ is a positive function,
		\item $\Gamma$ is weak* to weak*-continuous.
	\end{enumerate}
	Then there exists a positive operator $S\in \SC^1(\R^d)$ with $\tr(S)=1$ such that 
	\begin{equation*}
  \Gamma(f)=f\star S
\end{equation*}
for any $f\in L^{\infty}(\R^{2d})$.
\end{thm}
\begin{proof}
	  Before we begin, we remark that assumption (4) is exactly what we need to conclude that $\Gamma$ is the Banach space adjoint of some bounded linear operator $\Gamma_{\ast}:\tco \to L^1(\R^{2d})$. The existence of $\Gamma_{\ast}$ is how assumption (4) will be used, although it will not be explicitly mentioned in this brief outline.  \\
	The first step of the proof is to show that $\Gamma$ induces a bounded mapping $\Gamma:L^1(\R^{2d})\to \tco$. A calculation using all the assumptions of the proposition shows that for a positive $f\in L^{\infty}(\R^{2d})\cap L^1(\R^d)$ and a positive operator $T\in \tco$ we have
	\begin{equation*}
  \iint_{\R^{2d}} \tr(T\alpha_z(\Gamma(f))) \ dz =\iint_{\R^{2d}} f(z) \ dz \ \tr(T).
\end{equation*}
Comparing this with lemma \ref{lem:werner} we see that $\Gamma(f)\in \tco$ with $\|\Gamma(f)\|_{\tco}=\tr(\Gamma(f))=\iint_{\R^{2d}} f(z) \ dz$. This result holds for positive $f\in L^{\infty}(\R^{2d})\cap L^1(\R^{2d})$, and using this we may extend $\Gamma$ to a well-defined bounded operator from $L^1(\R^{2d})$ to $\tco$.

Using $\Gamma:L^1(\R^{2d})\to \tco$  we can construct a measure on $\R^{2d}$ with values in $\tco$. For a bounded, Lebesgue measurable subset $\Omega \subset \R^{2d}$, we define a measure by $\Omega\mapsto \Gamma(\chi_{\Omega})$. By our previous calculation we have $\|\Gamma(\chi_{\Omega})\|_{\tco}=\iint_{\R^{2d}} \chi_{\Omega}(z) \ dz=\mu(\Omega).$ This shows that our $\tco$-valued measure is absolutely continuous with respect to Lebesgue measure, and since $\tco$ has the Radon-Nikodym property \footnote{This follows from theorem 1 on page 79 of \cite{Diestel:1977}, as $\tco$ is the dual space of the compact operators.} it follows that there is some measurable $\bar{S}: \R^{2d}\to \tco$ such that \footnote{We have ignored one issue: we need to restrict $\Omega$ to bounded subsets to ensure that $\chi_{\Omega}\in L^1(\R^{2d})$. The technical details needed to circumvent this issue are given in the proof of lemma 3.1 in \cite{Kiukas:2006}.}
\begin{equation*}
  \Gamma(f)=\iint_{\R^{2d}} f(z) \bar{S}(z) \ dz.
\end{equation*}
The proof is now concluded by showing that assumption (2) and uniqueness of Radon-Nikodym derivatives imply that the function $\bar{S}(z)$ is given by $\bar{S}(z)=\alpha_z(S)$ for some fixed $S\in \tco$ -- see \cite{Werner:1984} or \cite{Kiukas:2006} for the details. 
\end{proof}
\begin{rem}
\begin{enumerate}
	\item Mappings $\Gamma:L^{\infty}(\R^{2d})\to \bo$ having these four properties are called \textit{positive correspondence rules} by Werner \cite{Werner:1984}.
	\item Recently, a similar result for $\Gamma: \mathcal{S}'(\Rdd)\to \tempdist$ has been proved by Cordero et al. \cite[Thm. 4.7]{Cordero:2018} at the level of Weyl symbols.
	\item The proof that mappings of the form $\Gamma(f)=f\star S$ are positive correspondence rules, for positive $S\in \tco$ with $\tr(S)=1$, is deferred to section \ref{sec:genlocandpovm} -- see the remark following proposition \ref{prop:covariantpovm}.
\end{enumerate}
\end{rem}
We claim that theorem \ref{thm:poscorrule} shows that our definition of mixed-state localization operators is natural. Consider the case of a localization operator $\mathcal{A}_{\Omega}^{\varphi}$ for $\Omega\subset \Rdd$ and $\varphi \in \HS$ with $\|\varphi\|_{L^2}=1$. These localization operators define a mapping $\Gamma_{\varphi}:L^{\infty}(\Rdd)\to \bo$ by $f\mapsto \mathcal{A}_{f}^{\varphi}=f\star (\varphi \otimes \varphi)$. For $\Gamma_{\varphi}$, the four properties in theorem \ref{thm:poscorrule} are true and have natural interpretations: 
\begin{enumerate}
	\item  We have that $\Gamma_{\varphi}(\chi_{\Rdd})\psi=\mathcal{A}_{\R^{2d}}^{\varphi}\psi=\psi$ for $\psi\in \HS$, which formalizes the fact that localizing $\psi\in L^2(\R^d)$ to the whole time-frequency plane $\R^{2d}$ should return $\psi$.    
	\item For a characteristic function $\chi_{\Omega}$, the property $$\Gamma_{\varphi}(T_z\chi_{\Omega})=\alpha_z(\Gamma_{\varphi}\chi_{\Omega})$$ says that
	$$\mathcal{A}_{\Omega-z}^{\varphi}=\mathcal{A}_{\Omega}^{\pi(z)\varphi} $$ 
	where $\Omega-z=\{z'-z:z'\in \Omega\}$. We may interpret this as saying that shifting the domain $\Omega$ of a localization operator by $z\in \Rdd$ is equivalent to replacing the window $\varphi$ with the time-frequency shift $\pi(z)\varphi$.   
	\item Since $\Gamma_{\varphi}(\chi_{\Omega})=\mathcal{A}_{\Omega}^{\varphi}$ is interpreted as an operator that picks out the part of a signal living in $\Omega$ in the time-frequency plane, it makes sense that $\inner{\mathcal{A}_{\Omega}^{\varphi}\psi}{\psi}\geq 0$ -- i.e. $\Gamma_{\varphi}(\chi_{\Omega})$ is a positive operator. 
	\item $\Gamma_{\varphi}:L^{\infty}(\Rdd) \to \bo$ is weak* to weak*-continuous, in particular assigning a localization operator $\mathcal{A}_{\Omega}^{\varphi}$ to a domain $\Omega$ is continuous in this sense. 
\end{enumerate}
It seems natural to require that a generalization of localization operators also satisfies (1), (2), (3) and (4), and theorem \ref{thm:poscorrule} shows that we are then lead to our definition of mixed-state localization operators. 

\subsection{Uniqueness of the domain} \label{sec:reconstructdomain}
In recent years the question of obtaining the domain $\Omega$ from the localization operator $\mathcal{A}_{\Omega}^{\varphi}$ has received some attention \cite{Abreu:2012,Abreu:2016}. In this section we will consider the theoretical possibility of such reconstruction for the mixed-state localization operators: if $S\in \tco$, when is the domain $\Omega \subset \Rdd$ uniquely determined by the mixed-state localization operator $\chi_{\Omega}\star S$, up to sets of Lebesgue measure zero\footnote{By "up to sets of Lebesgue measure zero" we mean that we regard two sets $\Omega,\Omega^\prime$ to be equal if $\mu(\Omega \triangle \Omega^{\prime})=0$, where $\triangle$ is the symmetric difference of sets.}? Since the localization operators $\mathcal{A}_{\Omega}^{\varphi}$ form a subset of the mixed-state localization operators, our results will also be applicable to such operators.
Our results follow from theorem \ref{thm:wiener} -- the Tauberian theorem for convolutions with operators. The first result concerns domains $\Omega$ with $\mu(\Omega)<\infty$.
\begin{thm}
\label{thm:boundedrec}
	\begin{enumerate} 
		\item  If $S\in \tco$ is such that the set $\{z\in \R^{2d}:\F_W(S)(z)=0\}$ has dense complement in $\R^{2d}$, then any $\Omega\subset \R^{2d}$  with finite Lebesgue measure is uniquely determined by the operator $\chi_{\Omega}\star S$, up to Lebesgue measure zero.
		\item  If $\varphi_1,\varphi_2 \in L^2(\R^d)$ are windows such that the set $\{z\in \R^{2d}: A(\varphi_2,\varphi_1)(z)=0\}$ has dense complement in $\R^{2d}$, then any $\Omega\subset \R^{2d}$  with finite Lebesgue measure is uniquely determined by the operator $\mathcal{A}_{\Omega}^{\varphi_1,\varphi_2}$, up to Lebesgue measure zero.
	\end{enumerate}
\end{thm}
\begin{proof}
	Follows from the implication $(1)\implies (3)$ in proposition \ref{prop:densityandreconstruction}.
	\end{proof} 
\begin{rem}
	In \cite{Abreu:2016} it is shown that the theory of accumulated spectrograms gives a method for reconstructing a compact domain $\Omega$, using the spectrograms of a finite subset of the eigenfunctions of $\mathcal{A}_{R\cdot \Omega}^{\varphi}$ as $R\to \infty$. Note, however, that this requires knowledge of $\mathcal{A}_{R\cdot \Omega}^{\varphi}$ as $R\to \infty$, and hence not merely of $\mathcal{A}^\varphi_{\Omega}$. On the other hand it is also shown in \cite{Abreu:2016} that $\chi_\Omega$ can be \textit{estimated} using only the spectrograms of a finite number of eigenfunctions of $\mathcal{A}^\varphi_{\Omega}$. In a coming work we use quantum harmonic analysis to show that this is possible for any mixed-state localization operator with compact domain. 
	\end{rem}
	
To the knowledge of the authors, the problem of reconstructing \textit{unbounded} domains $\Omega$ from localization operators $\mathcal{A}_{\Omega}^\varphi$ for $\varphi\in \HS$ has not previously been considered in the literature. We will cover the more general question of reconstructing an unbounded domain $\Omega$ from a mixed-state localization operator $\chi_{\Omega}\star S$. Since an unbounded set $\Omega$ may have infinite Lebesgue measure, we will not be able to use that $\chi_{\Omega}\in L^1(\R^{2d})$ as we did in the proof of the previous corollary. We need to consider $\chi_{\Omega}$ as an element of $L^{\infty}(\R^{2d})$. This leads to a stronger condition on the set of zeros of the Fourier-Wigner transform.

\begin{thm} \label{thm:rec}
		\begin{enumerate}
		\item  If $S\in \tco$ is such that the set $\{z\in \R^{2d}:\F_W(S)(z)=0\}$ is empty, then any measurable $\Omega\subset \R^{2d}$  is uniquely determined by the operator $\chi_{\Omega}\star S$, up to Lebesgue measure zero.
		\item  If $\varphi_1,\varphi_2 \in L^2(\R^d)$ are windows such that the set $\{z\in \R^{2d}: A(\varphi_2,\varphi_1)(z)=0\}$ is empty, then any measurable $\Omega\subset \R^{2d}$ is uniquely determined by the operator $\mathcal{A}_{\Omega}^{\varphi_1,\varphi_2}$, up to Lebesgue measure zero.
	\end{enumerate}
\end{thm}
\begin{proof}
	The proof is the same as in theorem \ref{thm:boundedrec}, except that we use proposition \ref{prop:densityandreconstruction2}.
\end{proof}
 
\section{Cohen's class and convolutions of operators} \label{sec:cohenand convolution}
In section \ref{sec:cohensclass} we defined Cohen's class to be those quadratic time-frequency representations $Q_{\phi}$ of the form
\begin{equation} \label{eq:cohenasweyl}
  Q_{\phi}(\psi)=W(\psi,\psi)\ast \phi
\end{equation}
for some $\phi\in \mathcal{S}^{\prime}(\Rdd)$ and any $\psi \in \mathcal{S}(\Rd)$. In this section we give a characterization of Cohen's class as convolutions with a fixed operator. We will show that many properties of the Cohen's class distribution may be precisely described as properties of the corresponding operator. 

\begin{prop} \label{prop:cohensclassasconvolutions}
	For $\phi \in \mathcal{S}^{\prime}(\Rdd)$, the associated Cohen's class distribution $Q_\phi$ is given by
\begin{equation} \label{eq:cohenasconvolutions}
  Q_{\phi}(\psi)= (\psi\otimes \psi) \star L_{\phi} \text{ \ for } \psi \in \mathcal{S}(\Rd),
\end{equation}
  where $L_{\phi}$ is the Weyl transform of $\phi$. \\
 Conversely, any operator $A\in \tempdist$ determines a Cohen's class distribution by 
 \begin{equation*} 
  Q_{A}(\psi):= (\psi\otimes \psi) \star \check{A}  \text{ \ for } \psi \in \mathcal{S}(\Rd).
\end{equation*}
\end{prop}

\begin{proof}
	We will apply the symplectic Fourier transform twice to equation \eqref{eq:cohenasconvolutions} and use parts (4) and (5) of proposition \ref{prop:schwartzoperators}. First note that
	\begin{equation*}
  \F_{\sigma}((\psi\otimes \psi) \star L_\phi)= \F_W(\psi\otimes \psi)\F_W(L_\phi)= A(\psi,\psi) \F_W(L_\phi)
\end{equation*}
using lemma \ref{lem:FWrankone}. We now apply $\F_{\sigma}$ again, and since $\F_{\sigma}\F_{\sigma}$ is the identity operator we find
\begin{align*}
  (\psi\otimes \psi) \star L_\phi&=\F_{\sigma}(A(\psi,\psi) \F_W(L_\phi)) \\
  &= \F_{\sigma}(A(\psi,\psi)) \ast \F_{\sigma}\F_W(L_\phi) \\
  &= W(\psi,\psi) \ast \phi=Q_\phi(\psi),
\end{align*}
where we have used that $\F_{\sigma}(A(\psi,\psi))=W(\psi,\psi)$ and that $\F_\sigma \F_W L_\phi$ is the Weyl symbol of $L_\phi$ by part (5) of proposition \ref{prop:schwartzoperators}. The second statement follows easily from the first: let $\phi$ be the Weyl symbol of $\check{A}$. The first part states that $Q_\phi(\psi)=(\psi\otimes \psi) \star L_{\phi}=(\psi\otimes \psi)\star \check{A}=Q_A(\psi),$ showing that $Q_A$ is of Cohen's class. 
\end{proof}

\begin{rem}
In light of lemma \ref{lem:grochenig}, this proposition shows that any shift-invariant \footnote{In the sense that $Q(\pi(z)\psi)=T_z(Q(\psi))$ for $z\in \Rdd$ and $\psi \in \HS$.}, weakly continuous quadratic time-frequency distribution is given by a convolution with a fixed operator on $\HS$.
\end{rem}
By proposition \ref{prop:cohensclassasconvolutions}, any Cohen's class distribution may be described by either a distribution $\phi \in \mathcal{S}^{\prime}(\Rdd)$ or by an operator $A\in \tempdist$, where 
\begin{equation*} 
	Q_\phi=Q_A \text{ if } L_\phi=\check{A}.
\end{equation*}
 We have defined $Q_A$ in terms of $\check{A}$ to simplify formulas in section \ref{sec:connecting}, and the reader should note that $A$ and $\check{A}$ share all relevant properties, such as positivity, trace and membership of Schatten classes. Using that $Q_{\phi}(\psi)=(\psi\otimes \psi)\star L_{\phi}$, we may apply the theory of convolutions of operators to deduce some simple results on Cohen's class distributions.   

\begin{prop} \label{prop:cohenandlp} 
	Fix $1\leq p \leq \infty$. Consider a Cohen's class distribution $Q_{\phi}$ for $\phi \in \mathcal{S}^{\prime}(\Rdd)$. Let $L_{\phi}$ be the Weyl transform of $\phi$. If $L_{\phi}\in \SC^p$, then $Q_{\phi}(\psi)$ is well-defined for any $\psi \in \HS$ and $Q_{\phi}(\psi)\in L^p(\Rdd)$ with $\|Q(\psi)\|_{L^p}\leq \|\psi\|_{L^2}^2 \|S\|_{\SC^p}$. In particular, if $L_{\phi}\in \bo$, then $Q_{\phi}(\psi)\in L^{\infty}(\Rdd)$ with $\|Q(\psi)\|_{L^{\infty}}\leq \|\psi\|_{L^2}^2 \|S\|_{\bo}$.
\end{prop}
  \begin{proof}
  	For any $\psi \in \HS$ we have that $\psi\otimes \psi\in \tco$ with $\|\psi\otimes \psi\|_{\tco}=\|\psi\|_{L^2}^2$. Since $Q_{\phi}(\psi)=(\psi\otimes \psi) \star L_{\phi}$ by proposition \ref{prop:cohensclassasconvolutions}, the results follow from proposition \ref{prop:convschatten}.
  \end{proof}
  \begin{rem}
  	By Pool's theorem \cite{Pool:1966}, the condition that $L_{\phi}\in \SC^2$ is equivalent to $\phi \in L^2(\Rdd)$. Unfortunately there is no equally simple characterization of those $\phi$ such that $L_{\phi}\in \tco$ or $L_{\phi}\in \bo$.
  \end{rem}
 
\begin{exmp} \label{exmp:cohen}
\begin{enumerate}
	\item The Wigner distribution $Q_{\phi}(\psi)=W(\psi,\psi)$ is given by $\phi=\delta_0$ in equation \eqref{eq:cohenasweyl}. By proposition \ref{prop:cohensclassasconvolutions}, $W(\psi,\psi)$ is also given by 
\begin{align*}
  W(\psi,\psi)=(\psi \otimes \psi) \star L_{\delta_0}
\end{align*}
for $\psi \in \mathcal{S}(\Rd)$. By a result of Grossmann, $L_{\delta_0}=2^dP$, where $P$ is the parity operator \cite{Grossmann:1976}. 
\item Fix a window $\varphi\in \HS$ and consider the operator $S=\varphi\otimes \varphi$. Then $\check{S}=\check{\varphi}\otimes \check{\varphi}$, and by proposition \ref{prop:cohensclassasconvolutions}, $S$ defines a Cohen's class distribution $Q_S$ by
\begin{equation*}
  Q_S(\psi)=(\psi \otimes \psi) \star (\check{\varphi}\otimes \check{\varphi}) = |V_{\varphi}\psi|^2,
\end{equation*}
where the last expression follows from lemma \ref{lem:spectrogramasconvolution}. This Cohen's class distribution is therefore the spectrogram. The corresponding function $\phi$, i.e. the Weyl symbol of $\check{\varphi}\otimes \check{\varphi}$, is the Wigner distribution $W(\check{\varphi},\check{\varphi})$.
\end{enumerate}

The idea of using operators to define time-frequency distributions appeared in the work of Wigner \cite{Wigner:1997}. Wigner assumed the existence of a self-adjoint operator $A(z)\in \bo$ for each $z\in \Rdd$, and then defined a distribution $Q_A(\psi)(z)=\inner{A(z)\psi}{\psi}$ for $\psi \in \HS$. As Janssen notes \cite{Janssen:1997}, it follows from lemma \ref{lem:grochenig} that the desirable property $Q_A(\pi(z')\psi)=T_{z'}Q_A(\psi)$ is only satisfied if $A(z)=\pi(z)A\pi(z)^{\ast}$ for some fixed operator $A\in \bo$. In this case we get by the definition of the convolution of two operators that $Q_A(\psi)=(\psi\otimes \psi)\star \check{A}$. However, this approach is not pursued any further than this remark in \cite{Janssen:1997}.

\end{exmp}

 \subsection{Positive Cohen's class distributions} \label{sec:cohenpositive}
 We say that a Cohen's class distribution $Q_{\phi}$ is \textit{positive} if $Q_{\phi}(\psi)(z)\geq 0$ for all $z\in \Rdd$ and $\psi$ in the domain of $Q_{\phi}$. As has been pointed out by Gr\"ochenig \cite[Ch. 14.6]{Grochenig:2001}, positivity of $Q_{\phi}$ may be expressed in terms of the corresponding operator $L_{\phi}$.
 \begin{prop} \label{prop:positivityofcohen}
 	Let $Q_{\phi}$ be a Cohen's class distribution such that the Weyl transform $L_{\phi}$ is bounded on $\HS$. Then $Q_{\phi}$ is positive if and only if $L_{\phi}$ is a positive operator.  
 \end{prop}
 \begin{proof}
 If $Q_{\phi}(\psi)$ is positive, then we have in particular for any $\psi \in \HS$ that 
 \begin{equation*}
  0\leq Q_{\phi}(\psi)(0)=\inner{\check{L_{\phi}}\pi(0)^* \psi}{\pi(0)^*\psi}=\inner{\check{L_{\phi}}\psi}{\psi},
\end{equation*}
where we have used that $(\psi\otimes \psi) \star L_{\phi}(z)=\inner{\check{L_{\phi}}\pi(z)^* \psi}{\pi(z)^*\psi}$, as follows from the definition of the convolution of operators.  This shows that $\check{L_{\phi}}$ is positive, hence $L_\phi$ is positive. We have used that the function $(\psi\otimes \psi) \star L_{\phi}(z)$ is in fact a continuous function \cite[Prop 3.3 (3)]{Luef:2017vs} to ensure that it has a well-defined value at $0$. If we assume that $L_{\phi}$ is positive, then we get that $Q_{\phi}$ is positive since $Q_{\phi}(\psi)=(\psi\otimes \psi)\star L_{\phi}$, $\psi\otimes \psi$ is a positive operator and the convolution of positive operators is a positive function by proposition \ref{prop:positiveandidentity}.
 \end{proof}
 The condition that a time-frequency distribution $Q_{\phi}$ should be positive is a natural requirement. One might therefore ask what conditions $\phi \in \mathcal{S}^{\prime}(\Rdd)$ should satisfy to ensure that $Q_{\phi}$ is positive. By the previous proposition, we may equivalently ask which conditions $\phi$ must satisfy to ensure that the Weyl transform $L_{\phi}$ is a positive operator. This question is of interest in quantum mechanics, and providing a general answer has turned out to be difficult. The so-called KLM conditions due to Kastler \cite{Kastler:1965}, and Loupias and Miracle-Sole \cite{Loupias:1966,Loupias:1967} are the most famous result of this kind, and we now formulate these conditions in our context, using the notation from \cite{deGosson:2011wq,Cordero:2017}. 
 \begin{thm}
 	Let $\phi$ be a real-valued function on $\Rdd$ such that the Weyl transform $L_{\phi}\in \tco$. The Cohen's class distribution $Q_{\phi}$ is positive if and only if
 	\begin{enumerate}
 		\item $\phi$ is continuous.
 		\item For every $N\geq 1$ and every $N$-tuple $(z_1,...,z_N)\in \left(\Rdd\right)^N$ the $N\times N$ matrix with entries
 		\begin{equation*}
  M_{jk}=e^{-2\pi i \sigma(z_j,z_k)}\F_{\sigma}(\phi)(z_j-z_k)
\end{equation*}
is positive semidefinite.
 	\end{enumerate} 
 \end{thm}
\begin{proof}
	The KLM-conditions state that if $\phi$ is a real-valued function on $\Rdd$ such that the Weyl transform $L_{\phi}\in \tco$, then the operator $L_{\phi}$ is positive if and only if $\phi$ satisfies the two properties above \cite[Prop 306 and Thm 309]{deGosson:2011wq}. By proposition \ref{prop:positivityofcohen}, $Q_{\phi}$ is positive if and only if $L_{\phi}$ is a positive operator.
\end{proof}
There are other, more recent results of this nature. Cordero, de Gosson and Nicola \cite{Cordero:2017} have recently proved a version of the KLM-conditions that seems more tractable for numerical verification. In  fact, their conditions characterize those $\phi \in L^2(\Rdd)$ such that the Weyl transform $L_{\phi} \in \SC^2$ is a positive operator. In terms of Cohen's class, their condition characterizes those $\phi \in L^2(\Rdd)$ such that $Q_{\phi}$ is positive. The reader is referred to \cite{Cordero:2017} for precise statements and proofs.

\subsection{Cohen's class distributions with the correct total energy property} \label{sec:cohentotalenergy}
 
Following Janssen \cite{Janssen:1997}, we say that a Cohen's class distribution $Q_{\phi}$ has the \textit{correct total energy property} if
\begin{equation*}
  \iint_{\Rdd} Q_{\phi}(\psi)(z) \ dz = \|\psi\|_{L^2}^2
\end{equation*}
for all $\psi\in \HS$. One might think of $Q_{\phi}(\psi)$ as an energy distribution for the signal $\psi$, and so one would hope that the total energy $\|\psi\|_{L^2}^2$ equals the integral of the energy distribution $Q_{\phi}(\psi)$. We now show how this property of the distribution $Q_{\phi}$ is related to properties of the Weyl transform $L_{\phi}$.

\begin{prop} \label{prop:moyalforcohen}
	Let $Q_{\phi}$ be a Cohen's class distribution, and let $L_{\phi}$ be the Weyl transform of $\phi$. If $L_{\phi}\in \tco$, then
	\begin{equation} \label{eq:integralformula}
  \iint_{\Rdd} Q_{\phi}(\psi) \ dz = \|\psi\|_{L^2}^2 \tr(L_{\phi})
\end{equation}
for any $\psi \in \HS$. If in addition $\phi \in \Ldd$ , then
\begin{equation} \label{eq:simpletraceformula}
	\iint_{\Rdd} Q_{\phi}(\psi) \ dz = \|\psi\|_{L^2}^2\iint_{\Rdd} \phi(z) \ dz
\end{equation}
\end{prop}
\begin{proof}
	By proposition \ref{prop:cohensclassasconvolutions}
	\begin{equation*}
  Q_{\phi}(\psi)=(\psi\otimes \psi)\star L_{\phi}.
\end{equation*}
By the definition of the convolution of two operators, $(\psi\otimes \psi)\star L_{\phi}=\tr((\psi\otimes \psi)\alpha_z(PL_{\phi}P))$. Since $L_{\phi}$ is assumed to be trace class, we may apply lemma \ref{lem:werner} to find that
\begin{align*}
   \iint_{\Rdd} \tr((\psi\otimes \psi)\alpha_z(PL_{\phi}P)) \ dz &= \tr(\psi\otimes \psi)\tr(PL_{\phi}P) \\
   &= \|\psi\|_{L^2}^2 \tr(P^2L_{\phi})=\|\psi\|_{L^2}^2 \tr(L_{\phi}).
\end{align*}
 We have used that $\tr(\psi\otimes \psi)=\|\psi\|_{L^2}^2$, which is a simple consequence of the definition of the trace. If $\phi \in \Ldd$, we may use the well-known relation
 \begin{equation*}
  \tr(L_{\phi})=\iint_{\Rdd} \phi(z) \ dz
\end{equation*}
between a distribution $\phi$ and its Weyl transform in this case to complete the proof \cite[Prop. 286]{deGosson:2011wq}. 
\end{proof}
\begin{rem} 
 There are many examples of Cohen's class distributions $Q_\phi$ where $L_\phi \in \tco$ yet $\phi \notin L^1(\Rdd)$, so that equation \eqref{eq:simpletraceformula} does not apply. For instance, if $\phi=W(\check{\varphi},\check{\varphi})$ for $\varphi \in L^2(\Rd)$, we saw in example \ref{exmp:cohen} that $Q_\phi$ is a spectrogram and $L_\phi=\check{\varphi}\otimes \check{\varphi} \in \tco$. For $W(\check{\varphi},\check{\varphi})\in L^1(\Rdd)$ to hold, $\varphi$ must be an element of the so-called \textit{Feichtinger algebra} \cite{Feichtinger:1981, deGosson:2011wq}, in particular $\varphi$ must be continuous. Hence equation \eqref{eq:integralformula} holds for a set of Cohen's class distributions where equation \eqref{eq:simpletraceformula} does not even make sense, and equation \eqref{eq:integralformula} and the connection to the trace of an operator is new to the best of our knowledge.
 
  In the special case of $\phi \in \mathcal{S}(\Rdd)$, equation \eqref{eq:simpletraceformula} is contained in section 2.4.2 of Janssen's survey \cite{Janssen:1997}. In this case both $L_{\phi}\in \tco$ and $\phi \in L^1(\Rdd)$ are satisfied, so Janssen needed neither equation \eqref{eq:integralformula} nor the connection to trace class operators. 
\end{rem}
A recurring theme in Janssen's thorough survey \cite{Janssen:1997} is that positivity for a Cohen's class distribution $Q_{\phi}$ is incompatible with many other desirable properties of $Q_{\phi}$. Proposition \ref{prop:moyalforcohen} shows that $Q_{\phi}$ may be both positive and have the correct total energy property if $L_{\phi}\in \tco$, but the next corollary shows that this fails spectacularly whenever $L_{\phi}\notin \tco$.
\begin{cor} \label{cor:energyinfinite}
	Let $Q_{\phi}$ a be a positive Cohen's class distribution, and let $L_{\phi}$ be the Weyl transform of $\phi$. If $L_{\phi} \in \bo \setminus \tco$, then
	\begin{equation*}
  \iint_{\Rdd} Q_{\phi}(\psi) \ dz = \infty
\end{equation*}
for all $\psi \in \HS$.
\end{cor}
\begin{proof}
	A simple approximation argument shows that the relation $\iint_{\Rdd} Q_{\phi}(\psi) \ dz = \|\psi\|^2_{L^2} \tr(L_{\phi})$ actually holds when $L_{\phi}$ is any bounded \textit{positive} operator on $\HS$, where $\tr(L_{\phi})=\infty$ if $L_{\phi}\notin \tco$. 
\end{proof}

\subsection{Characterization of positive Cohen's class distributions with correct total energy property} \label{sec:nicecohen}
The previous two subsections have introduced two desirable properties in a Cohen's class distribution $Q_{\phi}$, namely positivity and $\iint_{\Rdd} Q_{\phi}(\psi)(z) \ dz=\|\psi\|_{L^2}^2$ for any $\psi \in \HS$. Using the results from these subsections, we may now characterize those Cohen's class distributions with both these properties. In short, they are all given as linear combinations of the spectrogram.
\begin{thm} \label{thm:nicecohen}
Let $Q_{\phi}$ be a Cohen's class distribution. $Q_{\phi}$ is positive and has the correct total energy property
if and only if the Weyl transform $L_{\phi}$ is a positive trace class operator with $\tr(L_{\phi})=1$. If this is the case, there exists an orthonormal basis $\{\varphi_n\}_{n\in \N}$ in $\HS$ and a sequence $\{\lambda_n\}_{n\in \N}$ of positive numbers with $\sum_{n=1}^{\infty}\lambda_n=1$ such that
\begin{equation*}
  Q_{\phi}(\psi)(z) =\sum_{n=1}^{\infty}\lambda_n |V_{\varphi_n}\psi|^2(z),
\end{equation*}
and this sum converges uniformly for each $\psi\in \HS$. 
\end{thm}
\begin{proof}
	The main idea is to study $Q_{\phi}$ using the Weyl transform $L_{\phi}$, since $Q_{\phi}(\psi)=(\psi\otimes \psi)\star L_{\phi}$. Since $Q_{\phi}$ is positive, proposition \ref{prop:positivityofcohen} gives that $L_{\phi}$ is a positive operator. As we remarked in the proof of corollary \ref{cor:energyinfinite}, we then have that 
	\begin{equation*}
  \iint_{\Rdd} Q_{\phi}(\psi) \ dz = \|\psi\|^2_{L^2} \tr(L_{\phi}),
\end{equation*}
 for $\psi \in \HS$, where $\tr(L_{\phi})=\infty$ if $L_{\phi}$ is not trace class. We easily see from this expression that we need $L_{\phi}\in \tco$ with $\tr(L_{\phi})=1$ in order to have that $\iint_{\Rdd} Q_{\phi}(\psi)(z) \ dz=\|\psi\|_{L^2}^2$. Hence $L_{\phi}$ is a positive trace class operator, so we may use the singular value decomposition of $L_{\phi}$ to write
 \begin{equation*}
  L_{\phi}=\sum_{n=1}^{\infty} \lambda_n \varphi_n^{\prime} \otimes \varphi_n^{\prime}
\end{equation*}
where $\{\varphi_n^{\prime}\}_{n\in \N}$ is an orthonormal sequence in $\HS$ and $\{\lambda_n\}_{n\in \N}$ is a sequence of positive numbers with $\sum_{n=1}^{\infty}\lambda_n=\tr(L_{\phi})=1$. This sum of operators converges  to $L_{\phi}$ in the operator norm on $\bo$. In order to make the end results more aesthetic we define $\varphi_n=P\varphi_n^{\prime}$ for each $n\in \N$, so that $\varphi_n^{\prime}=\check{\varphi}_n$. The sequence $\{\varphi_n\}_{n\in \N}$ is clearly also orthonormal, and we have that
 \begin{equation*}
  L_{\phi}=\sum_{n=1}^{\infty} \lambda_n \check{\varphi}_n \otimes \check{\varphi}_n.
\end{equation*}
Now recall that $(\psi \otimes \psi)\star (\check{\varphi}_n \otimes \check{\varphi}_n)=|V_{\varphi_n}\psi|^2$ by lemma \ref{lem:spectrogramasconvolution} . For each $\psi \in \HS$ the operator $\psi \otimes \psi$ is trace class, and since the convolution of operators is continuous $\tco \times \bo\to L^{\infty}(\Rdd)$ by proposition \ref{prop:convschatten} we get that
\begin{align*}
  Q_{\phi}(\psi)&= L_{\phi}\star (\psi\otimes \psi) \\
  &= \left(\lim_{N\to \infty}\sum_{n=1}^{N} \lambda_n \check{\varphi}_n \otimes \check{\varphi}_n\right) \star (\psi\otimes \psi) \\
  &=\lim_{N\to \infty} \left(\sum_{n=1}^{N} \lambda_n \check{\varphi}_n \otimes \check{\varphi}_n\star (\psi\otimes \psi)\right)  \\
  &= \lim_{N\to \infty} \left(\sum_{n=1}^{N} \lambda_n |V_{\varphi_n}\psi|^2\right) 
\end{align*}
with convergence in the norm of $L^{\infty}(\Rdd)$, i.e. uniform convergence.
\end{proof}
 A restatement of the previous theorem is that the Cohen's class distributions $Q$ that are positive with the correct total energy property are exactly given by
\begin{equation*}
  Q(\psi)=(\psi\otimes \psi) \star S_Q
\end{equation*}
for some positive operator $S_{Q}\in \tco$ with $\tr(S_{Q})=1$. Operators of the form $S_Q$ are also known as density operators and play a central role in quantum mechanics, see \cite{deGosson:2017} for a modern and rigorous treatment. 

\begin{exmp}
	\begin{enumerate}
		\item The spectrogram $|V_{\varphi}\psi(z)|^2=(\psi \otimes \psi)\star (\check{\varphi}\otimes \check{\varphi})$ for $\varphi\in \HS$ with $\|\varphi\|_{L^2}=1$ is both positive and has the correct total energy property. This agrees with theorem \ref{thm:nicecohen} since the operator $\check{\varphi}\otimes \check{\varphi}$ is positive and $\tr(\check{\varphi}\otimes \check{\varphi})=\inner{\check{\varphi}}{\check{\varphi}}=1$.
		\item The Wigner distribution $W(\psi)=(\psi\otimes \psi) \star 2^d P$ is not positive by proposition \ref{prop:positivityofcohen}, as $P$ is not a positive operator. The correct total energy property holds for some, but not all $\psi \in \HS$ \cite{Grochenig:2001}.  
	\item Using a result due to Gracia-Bond\'ia and V\'arilly \cite{Gracia:1988}, we may now give a characterization of the Gaussians that give positive Cohen's class distributions with the correct total energy property. To make this precise, let $\Phi_M$ be the normalized Gaussian
	\begin{equation*}
  \Phi_M(z)=2^n \frac{1}{\det(M)^{1/4}}e^{-z^T\cdot  M \cdot z}    \text{ for $z\in \Rdd$},
\end{equation*}
where $M$ is a $2d\times 2d$-matrix. The result of \cite{Gracia:1988} states that the Weyl transform $L_{\Phi_M}$ is a positive trace class operator if and only if 
\begin{equation*}
  M=S^T \Lambda S,
\end{equation*}
where $S$ is a symplectic matrix and $\Lambda$ is diagonal matrix of the form \[\Lambda=\text{diag}(\lambda_1,\lambda_2,\dots,
\lambda_d,\lambda_1,\lambda_2,\dots,\lambda_d)\] with $0< \lambda_i\leq 1$. Hence these Gaussians $\Phi_M$ are exactly the Gaussians such that the Cohen's class distribution $Q_{\Phi_M}$ is positive with the correct total energy property. 
Note that this provides examples of positive Cohen's class distributions with the correct total energy property that are \textit{not} spectrograms, since some of the Gaussians above do not correspond to operators of the form $\check{\varphi}\otimes \check{\varphi}$ under the Weyl transform \cite{deGosson:2017,Gracia:1988}. These results are also linked with the symplectic structure of the phase space \cite{deGosson:2009}.
	\end{enumerate}
\end{exmp}

\subsection{Uncertainty principles for Cohen's class}
By using the connection between Cohen's class and convolutions of operators we obtain a weak uncertainty principle for Cohen's class distributions. The result is modeled on uncertainty principles for the spectrogram and Wigner distribution \cite{Grochenig:2001}.
\begin{prop}
	Let $S\in \bo$ and let $Q_S$ be the Cohen's class distribution determined by $Q_S(\psi)=(\psi\otimes \psi) \star S$ for $\psi \in \HS$. 
 If $\Omega\subset \Rdd$ is a measurable subset such that
	\begin{equation*}
  \iint_{\Omega}|Q_S(\psi)| \ dz \geq (1-\epsilon) \|S\|_{\bo}
\end{equation*}
for some $\psi\in \HS$ with $\|\psi\|_{L^2}=1$ and $\epsilon \geq 0$, then
\begin{equation*}
  \mu(\Omega)\geq 1-\epsilon.
\end{equation*}	
\end{prop}
\begin{proof}
By proposition \ref{prop:cohenandlp}, we know that $Q_S(\psi)\in L^{\infty}(\Rdd)$ with $\|Q_S(\psi)\|_{L^{\infty}}\leq\|\psi\|_{L^2}^2 \|S\|_{\bo}$. It follows that
\begin{equation*}
  \iint_{\Omega}|Q_S(\psi)| \ dz\leq \|\psi\|_{L^2}^2 \|S\|_{\bo} \iint_{\Omega} \ dz=\|\psi\|_{L^2}^2 \|S\|_{\bo} \mu(\Omega).
\end{equation*}
	Hence if 
	\begin{equation*}
  \iint_{\Omega}|Q_S(\psi)| \ dz \geq (1-\epsilon) \|\psi\|_{L^2}^2 \|S\|_{\bo},
\end{equation*}
we must have that 
\begin{equation*}
  \|\psi\|_{L^2}^2 \|S\|_{\bo} \mu(\Omega) \geq (1-\epsilon) \|\psi\|_{L^2}^2 \|S\|_{\bo},
\end{equation*}
and therefore $\mu(\Omega)\geq 1-\epsilon$.
\end{proof}
\subsection{Phase retrieval for Cohen's class distribution}
Given a Cohen's class distribution $Q_{\phi}$, one might ask whether any $\psi\in \HS$ is uniquely determined by $Q_{\phi}(\psi)$. Since proposition \ref{prop:cohensclassasconvolutions} shows that $\psi$ enters the expression for $Q_{\phi}(\psi)$ via $\psi\otimes \psi$, we can at most hope that $\psi\otimes \psi$ is uniquely determined by $Q_{\phi}(\psi)$. It is simple to show that $\psi_1\otimes \psi_1=\psi_2\otimes \psi_2$ if and only if $\psi_1=e^{i\alpha} \psi_2$ for some $\alpha\in \R$, so we will ask whether $\psi$ is determined by $Q_{\phi}(\psi)$ up to some constant phase $e^{i\alpha}$ with $\alpha\in \R$.  In the special case where $L_{\phi}\in \tco$, a rather weak condition on $\phi$ is enough to ensure this.
\begin{prop} \label{prop:reconstructingcohen}
	Let $\phi \in L^2(\Rdd)$ be a function such that the Weyl transform $L_{\phi}$ is trace class. Assume that the set $\{z\in \Rdd:\F_{\sigma}\phi(z)=0\}$ has dense complement in $\Rdd$. If $Q_{\phi}(\psi_1)=Q_{\phi}(\psi_2)$ for $\psi_1,\psi_2\in \HS$, then $\psi_1=e^{i \alpha}\psi_2$ for some constant $\alpha\in \R$.
\end{prop}
 \begin{proof}
 	By proposition \ref{prop:cohensclassasconvolutions}, we know that $Q_{\phi}(\psi)=(\psi \otimes \psi)\star L_{\phi}$. If the set $\{z\in \Rdd:\F_{W}L_{\phi}(z)=0\}$ has dense complement in $\Rdd$, we get from theorem \ref{thm:wiener} that the mapping $\psi\otimes \psi \mapsto (\psi\otimes \psi)\star L_{\phi}=Q_{\phi}(\psi)$ is injective. Hence, if $Q_{\phi}(\psi_1)=Q_{\phi}(\psi_2)$, then $\psi_1\otimes \psi_1=\psi_2\otimes \psi_2$. By the discussion preceding the proposition this implies that $\psi_1=e^{i\alpha}\psi_2$ for some constant $\alpha\in \R$. Furthermore, we know by proposition \ref{prop:fwspreading} that $\phi=\F_{\sigma}\F_W(L_{\phi})$, and applying $\F_{\sigma}$ to this equation we obtain $\F_{\sigma}\phi=\F_W(L_{\phi})$, so the sets $\{z:\in \Rdd:\F_{W}L_{\phi}(z)=0\}$ and $\{z:\in \Rdd:\F_{\sigma}\phi(z)=0\}$ are equal.
 \end{proof}
 When $\phi=W(\varphi,\varphi)$ for some $\varphi\in \HS$, the previous result gives a condition on the window $\varphi\in \HS$ to ensure that any $\psi\in \HS$ is determined by its spectrogram $|V_{\varphi}\psi|^2$ up to a phase $e^{i\alpha}$ for $\alpha\in \R$. This is the problem of \textit{phase retrieval} for the spectrogram \cite{Grohs:2017}.
\begin{cor} \label{cor:reconstructspectrogram}
 If $\varphi\in \HS$ and the set $\{z:\in \Rdd:A(\varphi,\varphi)(z)=0\}$ has dense complement in $\Rdd$ and $|V_{\varphi}\psi_1|=|V_{\varphi}\psi_2|$, then $\psi_1=e^{i \alpha}\psi_2$ for some constant $\alpha\in \R$.
\end{cor}
\begin{proof}
Let $\phi=W(\check{\varphi},\check{\varphi})$. As we saw in example \ref{exmp:cohen}, we then have 
\begin{equation*}
  Q_{\phi}(\psi)=|V_{\varphi}\psi|^2.
\end{equation*}
The result now follows from proposition \ref{prop:reconstructingcohen} by noting that $\F_{\sigma}W(\check{\varphi},\check{\varphi})=A(\check{\varphi},\check{\varphi})=\overline{A(\varphi,\varphi)}$, where the last equality follows from the definition of $A(\varphi,\varphi)$.
\end{proof}
\begin{rem}
	This corollary appears in \cite[Remark A.4]{Grohs:2017} under the stronger assumption that the set of zeros of the ambiguity function has Lebesgue measure 0. The same paper also proves that when $\varphi\in \mathcal{S}(\Rdd)$ and its ambiguity function has \textit{no} zeros, then the same result holds for $\psi_1,\psi_2 \in \mathcal{S}^{\prime}(\Rdd)$ \cite[Thm. 2.3]{Grohs:2017} -- a result that is referred to as "folklore" by \cite{Grohs:2017}.
\end{rem}
\section{Multiwindow STFTs and Cohen's class} \label{sec:connecting}
In sections \ref{sec:rigour} and \ref{sec:generalizedlocalization} we saw that an operator $S$ can be used to assign to a function $f$ on $\Rdd$ a multiwindow STFT-filter $f\star S$. On the other hand we saw in section \ref{sec:cohenand convolution} that $S$ defines a Cohen's class distribution $Q_S$ by $Q_S(\psi)(z)=((\psi \otimes \psi)\star \check{S})(z)$. In fact, there is a close connection between operators $f\star S$ and Cohen class distribution $Q_{\check{S}}$.

\begin{prop} \label{prop:cohenandmultiwindow}
	Let $S\in \tco$, $f\in L^{\infty}(\Rdd)$ and $\psi \in \HS$. Let $Q_{\check{S}}$ be the Cohen's class distribution $Q_{S}(\psi)=(\psi \otimes \psi)\star \check{S}$. Then
	\begin{equation} \label{eq:cohenandfilter1}
  \inner{f\star S}{\psi\otimes \psi}_{B(L^2),\tco}=\inner{f}{Q_{S}(\psi)}_{L^{\infty},L^1}. 
\end{equation}
  More explicitly
  \begin{equation} \label{eq:cohenandfilter2}
  \inner{(f\star S) \psi}{\psi}=\iint_{\Rdd} f(z) Q_{S}(\psi)(z) \ dz.
\end{equation}
\end{prop}
\begin{proof}
	When $f\in L^{\infty}(\Rdd)$, the operator $f\star S$ is defined by the relation
	\begin{equation*}
  \inner{f\star S}{T}_{B(L^2),\tco}=\inner{f}{\check{S}\star T}_{L^{\infty},L^1}
\end{equation*}
for any $T\in \tco$, as we noted in equation \eqref{eq:convolutionsofbounded}. In particular this must hold for $T=\psi \otimes \psi$. Since $Q_{S}(\psi)=\check{S}\star (\psi\otimes \psi)$, we get that 
\begin{equation*}
  \inner{f\star S}{\psi\otimes \psi}_{B(L^2),\tco}=\inner{f}{\check{S}\star (\psi \otimes \psi)}_{L^{\infty},L^1}=\inner{f}{Q_{S}(\psi)}_{L^{\infty},L^1}.
\end{equation*}
To prove equation \eqref{eq:cohenandfilter2} we recall that the duality action of $\bo$ on $\tco$ is given by $$\inner{f\star S}{\psi\otimes \psi}_{B(L^2),\tco}=\tr(f\star S(\psi\otimes \psi)).$$
By picking an orthonormal basis $\{e_n\}_{n\in \N}$ for $\HS$ we calculate that
\begin{align*}
  \tr((f\star S)(\psi\otimes \psi))&=\sum_{n\in \N} \inner{(f\star S)(\psi\otimes \psi)e_n}{e_n} \\
  &= \sum_{n\in \N} \inner{e_n}{\psi} \inner{(f\star S)\psi}{e_n} \\
  &= \inner{(f\star S) \psi}{\psi},
\end{align*}
where we have used Parseval's equality to remove the sum.
\end{proof}
\begin{rem}
The same result holds whenever $f\star S$ is defined in proposition \ref{prop:convschatten} and the brackets in equation \eqref{eq:cohenandfilter1} may be interpreted as duality. In the most general case we have $S \in \tempdist$, and we then have for $f\in \mathcal{S}(\Rdd)$ and $\psi \in \mathcal{S}(\Rd)$ that 
	\begin{equation*}
  \inner{(f\star S) \psi}{\psi}=\inner{f}{Q_{S}(\psi)}_{\mathcal{S}^\prime,\mathcal{S}}
\end{equation*}
where $Q_{S}$ is given by $Q_{S}(\psi)=(\psi \otimes \psi)\star \check{S}$.
\end{rem}
	The proposition shows that the Cohen's class distribution $Q_{S}$ has a naturally associated operator $f\star S$ for any $f\in L^{\infty}(\Rdd)$. The idea of associating operators to Cohen's class distributions has also been considered previously, but in this context it seems to be a novel insight that \textit{both the Cohen's class distribution and the associated operators are given by convolutions with a fixed operator} $S$ (and $\check{S}$).  Previous discussions of such results appear in \cite{Ramanathan:1994,Boggiatto:2010},  and more recently in \cite{Boggiatto:2017,Boggiatto:2018} where the operators $f\star S$ are called Cohen operators. In these references a Cohen's class distribution $Q$ was taken as a starting point, and it was shown that one could associate operators to $Q$ using a version of equation \eqref{eq:cohenandfilter2}.

Proposition \ref{prop:cohenandmultiwindow} generalizes several known relations between pseudodifferential operators and Cohen's class distributions.
\begin{exmp} \label{exmp:cohenoperators}
	\begin{enumerate}
		\item If we pick $S=\varphi \otimes \varphi$ for $\varphi\in L^2(\Rd)$, then $S\in \tco$ and $\check{S}=\check{\varphi}\otimes \check{\varphi}$. For $f\in L^{\infty}(\Rdd)$ the operator $f\star S$ is the localization operator $\mathcal{A}_{f}^{\varphi}$ by proposition \ref{lem:locasconv}. The Cohen's class distribution determined by $S$ is the spectrogram by example \ref{exmp:cohen}
		\begin{equation*}
  Q_S(\psi)(z)=(\psi\otimes \psi) \star (\check{\varphi}\otimes \check{\varphi})(z)=|V_{\varphi}\psi(z)|^2.
\end{equation*}
Equation \eqref{eq:cohenandfilter2} states the familiar relation
\begin{equation*}
  \inner{\mathcal{A}_{f}^{\varphi}\psi}{\psi}= \iint_{\Rdd} f(z)|V_{\varphi}\psi(z)|^2 \ dz.
\end{equation*}
\item For $S=2^dP$ the proposition describes the Weyl calculus. As we observed in example \ref{exmp:cohen} the Cohen's class distribution associated to $2^dP=\left(2^dP\right){\check{\phantom{x}}}$ is the Wigner distribution
\begin{equation*}
  Q_{2^d P}(\psi)=(\psi\otimes \psi) \star 2^dP(z) = W(\psi)(z). 
\end{equation*}
For a function $f\in L^1(\Rdd)$ the operator $f\star 2^dP$ is the Weyl transform $L_f$ of $f$: the Weyl symbol of $2^dP$ is $\delta_0$\cite{Grossmann:1976} and hence the Weyl symbol of $f\star 2^dP$ is $f$ by proposition \ref{prop:weylsymbolofmultiwindow}. Equation \eqref{eq:cohenandfilter2}  becomes
\begin{equation*}
  \inner{L_f \psi}{\psi}=\iint_{\Rdd} f(z) W(\psi)(z) \ dz
\end{equation*}
which is the equation we used to define the Weyl transform $L_f$.
\item When $\phi=\F_{\sigma}\Theta$, where $\Theta(z)=\frac{\sin(\pi x\omega)}{\pi x\omega}$, the Cohen's class distribution $Q_{\phi}$ is closely related to Born-Jordan quantization \cite{Cohen:1966,deGosson:2011}. By proposition \ref{prop:cohensclassasconvolutions} we may write $Q_{\phi}(\psi)=(\psi\otimes \psi)\star L_{\phi}$, where $L_{\phi}$ is the Weyl transform of $\phi$.

 For $f\in \mathcal{S}(\Rdd)$ we get the associated operators 
$
 f\star L_{\phi}.
$
In fact, $f\star L_{\phi}$ is the Born-Jordan quantization of the function $f$. To prove this, we note that in \cite{Cordero:2017} the Born-Jordan quantization of $f$ is defined to be the operator with Weyl symbol $f\ast \phi$. By proposition \ref{prop:weylsymbolofmultiwindow} the Weyl symbol of $f\star L_{\phi}$ is $f\ast \phi$, so $f\star L_{\phi}$ really is the Born-Jordan quantization of $f$. \\
Equation \eqref{eq:cohenandfilter1} is a well-known relation between Born-Jordan quantization and the Cohen's class distribution determined by $\phi$, in fact this is used to define Born-Jordan quantization in \cite{deGosson:2016}.
\end{enumerate}
\end{exmp}
\subsection{The localization problem for Cohen's class} 
The previous section showed that an operator $S$ allows the construction of operators $f\star S$ and a Cohen's class distribution $Q_{S}(\psi)=(\psi\otimes \psi)\star \check{S}$, and that the operators $f\star S$ are related to $Q_{S}(\psi)$ in a natural way. We will now consider this relationship when $f$ is the characteristic function $\chi_{\Omega}$ of some domain $\Omega \subset \Rdd$. In this case equation \eqref{eq:cohenandfilter2} from the previous section becomes
\begin{equation} \label{eq:cohenandlocalization}
    \inner{\chi_{\Omega}\star S \psi}{\psi}=\iint_{\Omega} Q_{S}(\psi)(z) \ dz.
\end{equation}
The right hand side of this equation may be interpreted as a measure of the concentration of the energy of $\psi$ in the region $\Omega$ of the time-frequency plane, and leads to a natural \textbf{localization problem for Cohen's class}\cite{Lieb:2010,Ramanathan:1993,Ramanathan:1994,Flandrin:1988} : \textit{for a Cohen's class distribution $Q$ and a measurable $\Omega \subset \Rdd$. Find the signal $\psi \in \HS$ with $\|\psi\|_{L^2}=1$ that maximizes} $$\iint_{\Omega} Q(\psi)(z) \ dz.$$
Equation \eqref{eq:cohenandlocalization} implies that the problem is solved by considering the eigenfunctions of the operator $\chi_{\Omega}\star S$ by Courant's min-max principle \cite[Thm. 28.4]{Lax:2002}, as the next proposition makes formal. 
\begin{prop} \label{prop:minmax}  
  Let $\Omega \subset \Rdd$ be a measurable subset, let $S\in \bo$ be a selfadjoint operator and let $Q_{S}$ be the associated Cohen's class distribution $Q_{S}(\psi)=(\psi \otimes \psi)\star \check{S}$. Assume that $\chi_{\Omega}\star S$ is a compact operator. Let $\lambda_1\geq \lambda_2,...$ be the positive eigenvalues of $\chi_{\Omega}\star S$ (counted with multiplicities) and let $\phi_i$ be the eigenvector corresponding to $\lambda_i$ for $i\in \N$. Then
  \begin{equation*}
  \iint_{\Omega} Q_{S}(\phi_n)(z) \ dz = \max \left\{\iint_{\Omega} Q_{S}(\psi)(z) \ dz: \|\psi\|_{L^2}=1, \psi\perp \phi_1,\phi_2,...,\phi_{n-1}\right\}.
\end{equation*}
  \end{prop}
  \begin{proof}
  By the min-max principle \cite[Thm. 28.4]{Lax:2002} we know that
  \begin{equation} \label{eq:minmaxproof1}
  \lambda_n=\min_{\psi_1,...,\psi_{n-1}} \max_{\substack{\psi\perp \psi_1,...,\psi_{n-1}\\ \|\psi\|_{L^2}=1}}  \inner{(\chi_{\Omega}\star S) \psi}{\psi}, 
\end{equation}
where $\psi_1,\psi_2,...,\psi_{n-1}$ is any set of linearly independent vectors in $\HS$.
Since $\lambda_n=\inner{(\chi_{\Omega}\star S) \phi_n}{\phi_n}$ and $\phi_n\perp \phi_1,...\phi_{n-1}$, the minimum in equation \eqref{eq:minmaxproof1} is achieved when $\psi_1=\phi_1,\psi_2=\phi_2,...,\psi_{n-1}=\phi_{n-1}$, hence
\begin{equation} \label{eq:minmaxproof2}
  \lambda_n= \max_{\substack{\psi\perp \phi_1,...,\phi_{n-1}\\ \|\psi\|_{L^2}=1}}  \inner{(\chi_{\Omega}\star S) \psi}{\psi}.
\end{equation}
By equation \eqref{eq:cohenandlocalization} we know that $\inner{(\chi_{\Omega}\star S) \psi}{\psi}=\iint_{\Omega} Q_{S}(\psi)(z) \ dz$, and since $\lambda_n=\inner{(\chi_{\Omega}\star S) \phi_n}{\phi_n}$ equation \eqref{eq:minmaxproof2} states that
\begin{equation*}
  \iint_{\Omega} Q_{S}(\phi_n)(z) \ dz = \max \left\{\iint_{\Omega} Q_{S}(\psi)(z) \ dz: \|\psi\|_{L^2}=1, \psi\perp \phi_1,\phi_2,...,\phi_{n-1}\right\}.
\end{equation*}
  \end{proof}
  \begin{rem}
  We have formulated the result by requiring that $\chi_{\Omega} \star S$ is compact. It is easy to find conditions making this true; by proposition \ref{prop:convschatten} it will be true if $\mu(\Omega)<\infty$ and $S\in \SC^p$ for some $p<\infty$. However, $\chi_{\Omega} \star S$ may well be compact in other cases too. 
\end{rem}
This idea of solving the localization problem by considering eigenfunctions of operators goes back to the work of Flandrin \cite{Flandrin:1988} for the Wigner distribution. Ramanathan and Topiwala \cite{Ramanathan:1994} later showed that similar techniques were possible for other Cohen's class distributions, by defining the operators $\chi_{\Omega}\star S$ in equation \eqref{eq:cohenandlocalization} using the Weyl calculus. Recently Boggiatto et. al have considered the same problem in \cite{Boggiatto:2017} using methods very similar to those we consider, but without the convolutions of operators and functions.
\begin{exmp}
	\begin{enumerate}
		\item If $S=\varphi \otimes \varphi$ for $\varphi\in \HS$, then we know from examples \ref{exmp:cohen} and \ref{exmp:cohenoperators} that $Q_{S}(\psi)=|V_{\varphi}\psi|^2$, the spectrogram, and $\chi_{\Omega}\star S$ is the localization operator $\mathcal{A}_{\Omega}^{\varphi}$. Proposition \ref{prop:minmax} says that the functions $\psi$ that minimize
		\begin{equation*}
  \iint_{\Omega} |V_{\varphi}\psi|^2(z) \ dz
\end{equation*}
are the eigenfunctions of the operator $\mathcal{A}_{\Omega}^{\varphi}$. This relation is well known \cite{Ramanathan:1994spec}, and exploited in the recent work of Abreu et al. on accumulated spectrograms \cite{Abreu:2012,Abreu:2016}.
\item When $S=2^dP$, we have seen in examples \ref{exmp:cohen} and \ref{exmp:cohenoperators} that $\chi_{\Omega}\star 2^dP$ is the Weyl transform $L_{\chi_{\Omega}}$ and that $Q_{S}(\psi)=W(\psi)$ -- the Wigner distribution of $\psi$. If we wish to find the functions $\psi\in \HS$ whose Wigner distributions is maximally concentrated in a domain $\Omega \subset \Rdd$, proposition \ref{prop:minmax} \footnote{The proposition requires that $\chi_{\Omega}\star 2^dP=L_{\chi_{\Omega}}$ is compact. Even though $P$ is not a compact operator, the operator $L_{\chi_{\Omega}}$ is compact whenever $\mu(\Omega)<\infty$, since $\chi_{\Omega}\in L^2(\Rdd)$ in this case and $L_f\in \SC^2$ whenever $f\in L^2(\Rdd)$ by Pool's theorem \cite{Pool:1966}.} reduces the problem to finding the eigenfunctions of the Weyl transform $L_{\chi_{\Omega}}$. This insight was first formulated in Flandrin's paper \cite{Flandrin:1988}, and extensions of his results include \cite{Ramanathan:1993} and \cite{Lieb:2010}.
	\end{enumerate}
\end{exmp}

Although proposition \ref{prop:minmax} only assumes that $\chi_{\Omega}\star S$ is compact and selfadjoint, the interpretation of 
\begin{equation*}
  \iint_{\Omega} Q_{\check{S}}(\psi)(z) \ dz
\end{equation*}
as the energy concentration of $\psi$ in $\Omega$ is more natural when $Q_{\check{S}}$ is positive and normalized in the sense that $$ \iint_{\Rdd} Q(\psi)(z) \ dz=\|\psi\|_{L^2}^2.$$ As we observed in section \ref{sec:nicecohen}, this is satisfied exactly when $S\in \tco$ is a positive operator with $\tr(S)=1$. In this case the operators $\chi_{\Omega}\star S$ are the mixed-state localization operators introduced in section \ref{sec:generalizedlocalization} using different arguments, and the next section considers this special case in detail.

\section{Localization operators and positive operator valued measures} \label{sec:genlocandpovm}
In this section we will approach the mixed-state localization operators $\chi_{\Omega}\star S$ from another perspective, namely that of covariant positive operator valued measures (POVMs). This perspective has been ever-present when the convolutions of operators have been introduced and discussed in quantum physics \cite{Holevo:1979, Werner:1984, Kiukas:2012gt,Kiukas:2006}, and we wish to show that it may be of interest also in time-frequency analysis. A POVM $F$ gives two possible measures of the time-frequency content of a signal $\psi$ in a domain $\Omega$ in the the time-frequency plane. On the one hand, the \textit{signal} $F(\Omega)\psi$ may be interpreted as the component of $\psi$ with time-frequency components in $\Omega$. On the other hand, we know from section \ref{sec:povm} that $\psi$ induces a probability measure $\mu_{\psi}^F$, and the \textit{number} $\mu_{\psi}^F(\Omega)$ measures the time-frequency content of $\psi$ in $\Omega$. \\
Given a signal $\psi\in \HS$, we wish to show that these two ways of measuring the time-frequency content of $\psi$ in a domain $\Omega$ lead to the study of mixed-state localization operators and Cohen's class distributions, respectively. The first step in this direction is to note that mixed-state localization operators define POVMs.

\begin{prop} \label{prop:genlocPOVM}
	Let $S\in \tco$ be a positive operator with $\tr(S)=1$. Then $S$ defines a covariant POVM $F$ by
	\begin{equation*}
  F(\Omega)=\chi_{\Omega}\star S
\end{equation*}
 for any measurable $\Omega \subset \Rdd$.
\end{prop}
\begin{proof}
	We get from proposition \ref{prop:positiveandidentity} that $F(\Omega)\geq 0$ and $F(\Rdd)=I$. The covariance of $F$ follows from the relation $\alpha_z(f\star S)=(T_zf)\star S$ for $f\in L^{\infty}(\Rdd)$ \cite{Skrettingland:2017}, since $\alpha_z(F(\Omega))=\alpha_z(\chi_{\Omega}\star S)=(T_z\chi_{\Omega})\star S=\chi_{\Omega+z}\star S=F(\Omega+z)$. \\
	If $\{\Omega_i\}_{i\in \N}$ is a collection of disjoint, measurable subsets of $\Rdd$ and $\Omega:=\cup_{i\in \N}\Omega_i$, then $\chi_{\Omega}=\sum_{i=1}^{\infty}\chi_{\Omega_i}$, where the sum converges pointwise. We need to show that $\chi_{\Omega}\star S=\sum_{i=1}^{\infty}\chi_{\Omega_i}\star S$ with convergence in the weak operator topology, i.e. $\sum_{i=1}^{\infty}\inner{\chi_{\Omega_i}\star S \phi}{\psi}=\inner{\chi_{\Omega}\star S \phi}{\psi}$ for all $\phi,\psi \in \HS$. Since $\chi_{\Omega}\in L^{\infty}(\Rdd)$, we know from  equation \eqref{eq:convolutionsofboundedexplicit} in section \ref{sec:werner} that the operator $\chi_{\Omega}\star S\in \bo$ is defined by the duality relation 
\begin{equation*}
  \tr(T(\chi_{\Omega}\star S))=\iint_{\Rdd} \chi_{\Omega}(z) \check{S}\star T(z) \ dz
\end{equation*}
for any $T\in \tco$. In particular, with $T=\phi \otimes \psi$ with $\phi,\psi \in \HS$ we get that
\begin{equation} \label{eq:fubiniandpovm}
  \inner{\chi_{\Omega}\star S\phi}{\psi}=\iint_{\Rdd} \chi_{\Omega}(z) (\check{S}\star (\phi\otimes \psi))(z) \ dz.
\end{equation} 
This implies that
\begin{equation*}
  \sum_{i=1}^{\infty} \inner{\chi_{\Omega_i}\star S\phi}{\psi}=  \sum_{i=1}^{\infty}\iint_{\Rdd} \chi_{\Omega_i}(z) (\check{S}\star (\phi\otimes \psi))(z) \ dz.
\end{equation*}
Since $\sum_{i=1}^{\infty}\chi_{\Omega_i}=\chi_{\Omega}$ and $\check{S}\star (\phi\otimes \psi)\in L^1(\Rdd)$ by proposition \ref{prop:convschatten}, we may use Fubini's theorem to change the order of integration, and we obtain that
\begin{align*}
  \sum_{i=1}^{\infty} \inner{\chi_{\Omega_i}\star S\phi}{\psi}&=\iint_{\Rdd}\sum_{i=1}^{\infty} \chi_{\Omega_i}(z) (\check{S}\star (\phi\otimes \psi))(z) \ dz \\
  &= \iint_{\Rdd} \chi_{\Omega}(z) (\check{S}\star (\phi\otimes \psi))(z) \ dz \\
  &= \inner{\chi_{\Omega}\star S\phi}{\psi},
\end{align*}
where the final line follows from equation \eqref{eq:fubiniandpovm}.
\end{proof}
In particular, this result implies that the localization operators $\mathcal{A}_{\Omega}^{\varphi}$ may be interpreted as POVMs.
\begin{cor}
	Let $\varphi\in \HS$ be a window with $\|\varphi\|_2=1$. Then $\varphi$ defines a POVM $F$ by
	\begin{equation*}
  F(\Omega)=\mathcal{A}_{\Omega}^{\varphi}.
\end{equation*}
\end{cor}
\begin{proof}
	Follows from lemma \ref{lem:locasconv} and the previous proposition with $S=\varphi \otimes \varphi$.
\end{proof}
\begin{rem}
	The fact that a localization operator determines a POVM has been remarked by other authors, such as \cite{Alanga:2014}.
\end{rem} 

\subsection{Cohen's class and POVMs}
By proposition \ref{prop:genlocPOVM}, a positive operator $S\in \tco$ with $\tr(S)=1$ defines a POVM $F$ via the mixed-state localization operators $F(\Omega)=\chi_{\Omega}\star S$. Once we have a POVM $F$, we know from section \ref{sec:povm} that we obtain a probability measure $\mu_{\psi}^F$ for each $\psi \in \HS$. As we mentioned at the start of the section, the time-frequency content of $\psi$ in $\Omega$ may be measured either by $F(\Omega)\psi$, or by the induced probability measure $\mu_{\psi}^F$. When the POVM $F$ is of the form in proposition \ref{prop:genlocPOVM}, the measures $\mu_{\psi}^{F}$ are given by the positive Cohen's class distribution induced by $S$ as in proposition \ref{prop:cohenandmultiwindow}.
\begin{lem} \label{lem:radonnikodym}
	Let $S\in \tco$ be a positive operator with $\tr(S)=1$, and consider the POVM $F(\Omega)=\chi_{\Omega}\star S$. For $\psi \in \HS$, the induced probability measure $\mu_{\psi}^F$ on $\R^{2d}$ is given by
	\begin{equation*}
  \mu_{\psi}^F(\Omega)=\iint_{\Omega}\left((\psi\otimes \psi) \star \check{S}\right)(z)\ dz.
\end{equation*}
In other words, the Radon-Nikodym derivative of $\mu_{\psi}^F$ w.r.t. Lebesgue measure $dz$ is the Cohen class distribution
\begin{equation*}
  Q_S(\psi)=(\psi\otimes \psi)\star \check{S}.
\end{equation*}

\end{lem}
\begin{proof}
	This is merely a restatement of proposition \ref{prop:cohenandmultiwindow} in the terminology of POVMs, since $\mu_{\psi}^F$ is defined by $\mu_{\psi}^F(\Omega)=\inner{F(\Omega)\psi}{\psi}=\inner{(\chi_{\Omega}\star S)\psi}{\psi}$.
\end{proof}
Since $S$ is assumed to be a positive operator with $\tr(S)=1$, we know that the Cohen's class distribution $Q_S$ in lemma \ref{lem:radonnikodym} is positive and has the correct total energy property by theorem \ref{thm:nicecohen}. This is exactly what we need to get that $\mu_{\psi}^{F}$ is a probability measure. 
\begin{exmp}
	Let $S=\varphi \otimes \varphi$ for $\varphi\in \HS$ with $\|\varphi\|_{L^2}=1$. As we have seen in example \ref{exmp:cohenoperators}, the POVM $F(\Omega)=\chi_{\Omega}\star S$ is given by the localization operators $F(\Omega)=\mathcal{A}_{\Omega}^\varphi$, and the associated Cohen's class distribution is the spectrogram: $Q_{S}(\psi)=|V_\varphi \psi|^2$. By lemma \ref{lem:radonnikodym} the induced probability measures $\mu_\psi^F$ are given by
	\begin{equation*}
  \mu_{\psi}^F(\Omega)=\iint_{\Omega} |V_{\varphi}\psi(z)|^2 \ dz,
\end{equation*}
hence the Radon-Nikodym derivatives of the probability measures induced by localization operators are spectrograms.
\end{exmp}
Another way of stating the relation between the mixed-state localization operators and the POVM $F$ that they induce, is to express the localization operators as an integral over the POVM.
\begin{prop} \label{prop:convolutionasintegralofpovm}
	Let $F$ be a POVM given by $F(\Omega)=\chi_{\Omega}\star S$ for some positive $S\in \tco$ with $\tr(S)=1$. Then
\begin{equation*}
  f\star S=\iint_{\Rdd} f dF.
\end{equation*}
In particular, the mixed-state localization operators $\chi_{\Omega}\star S$ may be expressed as
\begin{equation*}
  \chi_{\Omega}\star S=\iint_{\Omega} dF.
\end{equation*}

\end{prop}
\begin{proof}
	The operator $\iint_{\Rdd} f dF$ is by definition the unique operator satisfying
	\begin{equation*}
  \inner{\iint_{\Rdd} f dF \psi}{\psi}=\iint_{\Rdd} f d\mu_{\psi}^F
\end{equation*}
for each $\psi \in \HS$. We need to show that $f\star S$ satisfies this condition.
\begin{align*}
  \inner{f\star S\psi}{\psi}&= \inner{\iint_{\Rdd}f(z) \alpha_z(S) \ dz \ \psi}{\psi} \\
  &=\iint_{\Rdd}f(z) \inner{\alpha_z(S)\psi}{\psi} \ dz \\
  &= \iint_{\Rdd} f(z) \left((\psi\otimes \psi) \star \check{S}\right)(z) \ dz \\
  &= \iint_{\Rdd} f d\mu_{\psi}^F.
\end{align*}
In the calculation we have moved the inner product inside the integral. This is an instance of the definition of $f\star S$ in equation \eqref{eq:convolutionsofboundedexplicit}, when $T=\psi\otimes \psi$. We have also used the equality $\inner{\alpha_z(S)\psi}{\psi}=\left((\psi\otimes \psi) \star \check{S}\right)(z)$, which follows from the definition of the convolution of two operators. In the last line we used lemma \ref{lem:radonnikodym}.
\end{proof}

\begin{cor}
	Let $F$ be a POVM given by $F(\Omega)=\mathcal{A}_{\Omega}^{\varphi}$ for some window $\varphi\in \HS$ with $\|\varphi\|_2=1$. Then 
	\begin{equation*}
  \mathcal{A}_f^{\varphi}=\iint_{\Rdd} f dF.
\end{equation*}
In particular, the localization operators $\mathcal{A}_{\Omega}^{\varphi}$ may be expressed as
\begin{equation*}
  \mathcal{A}_{\Omega}^{\varphi} =\iint_{\Omega} dF.
\end{equation*}
\end{cor}

A much deeper result than proposition \ref{prop:genlocPOVM} is that the converse is also true: any covariant POVM $F$ is of the form in proposition \ref{prop:genlocPOVM} \cite{Holevo:1979,Werner:1984,Kiukas:2006}. We provide the proof in  our terminology for completeness.
\begin{prop} \label{prop:covariantpovm} 
	Let $F$ be a covariant POVM. There exists some positive $S\in \tco$ with $\tr(S)=1$ such that 
	\begin{equation*}
  F(\Omega)=\chi_{\Omega} \star S
\end{equation*}
for all $\Omega \subset \Rdd$.
\end{prop}
\begin{proof}
	We will show that the map $\Gamma:L^{\infty}(\Rdd)\to \bo$ defined by
	\begin{equation*}
  \Gamma(f)=\iint_{\Rdd}f dF
\end{equation*}
satisfies the conditions of theorem \ref{thm:poscorrule}. By that theorem we could then conclude that there is some positive $S\in \tco$ with $\tr(S)=1$ such that
\begin{equation*}
  \iint_{\Rdd} f dF = f \star S
\end{equation*}
 for any $f\in L^{\infty}(\Rdd)$, and in particular $F(\Omega)=\iint_{\Rdd} \chi_{\Omega} dF = \chi_{\Omega} \star S$. We check that the conditions in theorem \ref{thm:poscorrule} are satisfied.
 \begin{enumerate}
 	\item $\Gamma(\chi_{\Rdd})=\iint_{\Rdd} \chi_{\Rdd} dF=F(\Rdd)=I$, by the definition of a POVM. 
 	\item Fix $z^{\prime}\in \Rdd$ and $f\in L^{\infty}(\Rdd)$. We need to show that $\Gamma(T_{z^{\prime}}f)=\alpha_{z^{\prime}}(\Gamma(f))$, and by the uniqueness part of lemma \ref{lem:povmintegration} it suffices to show that
		 	\begin{equation} \label{eq:covariance}
		  \iint_{\Rdd} T_{z^{\prime}}f \ d\mu_{\psi}^F=\inner{\alpha_{z^{\prime}}(\Gamma(f))\psi}{\psi}.
		\end{equation}
		Note that
		\begin{align*}
		  \inner{\alpha_{z^{\prime}} (\Gamma(f))\psi}{\psi}&= \inner{\pi(z^{\prime})\Gamma(f)\pi(z^{\prime})^*\psi}{\psi} \\
		  &= \inner{\Gamma(f)\pi(z^{\prime})^*\psi}{\pi(z^{\prime})^*\psi} \\
		  &= \iint_{\Rdd} f \ d\mu_{\pi(z^{\prime})^*\psi}^F
		\end{align*}
		by lemma \ref{lem:povmintegration}. From the definition of the probability measure $\mu_{\pi(z^{\prime})^* \psi}^{F}$ and the covariance of $F$ we find that
		\begin{align*}
		  \mu_{\pi(z^{\prime})^* \psi}^{F}(\Omega)&=\inner{F(\Omega)\pi(z^{\prime})^* \psi}{\pi(z^{\prime})^* \psi} \\
		  &= \inner{\alpha_{z^{\prime}}(F(\Omega))\psi}{\psi} \\
		  &= \inner{(F(\Omega+z^{\prime}))\psi}{\psi} \\
		  &= \mu_{\psi}^F(\Omega+z^{\prime}).
		\end{align*}
		Hence 
		\begin{equation*}
		  \iint_{\Rdd} f \ d\mu_{\pi(z^{\prime})^*\psi}^F=\iint_{\Rdd} T_{z^{\prime}} f \ d\mu_{\psi}^F
		\end{equation*}
		by a change of variable, which proves equation \eqref{eq:covariance}.
 	\item By lemma \ref{lem:povmintegration}, the operator $\Gamma(f)=\iint_{\Rdd} f \ dF$ satisfies 
 		\begin{equation*}
  			\inner{\Gamma(f) \psi}{\psi}=\iint_{\Rdd} f(z) \ d\mu_{\psi}^F.
		\end{equation*}
		If $f$ is positive, the integral on the right hand side is clearly positive. Hence $\inner{\Gamma(f) \psi}{\psi}\geq 0$ for all $\psi \in \HS$, so $\Gamma(f)$ is a positive operator.
	\item To show that $\Gamma$ is weak*-weak*-continuous, we assume that a sequence $\{f_n\}_{n\in \N}$ in $L^{\infty}(\Rdd)$ converges in the weak*-topology to some $f\in L^{\infty}$. We need to show that $\tr(S\Gamma(f_n))$ converges to $\tr(S\Gamma(f))$ for each $S\in \tco$, since $\bo$ is the dual space of $\tco$. For each $S\in \tco$, the expression $\mu_{S}^{F}(\Omega)=\tr(SF(\Omega))$ defines a complex, finite measure on $\Rdd$, and lemma \ref{lem:povmintegration} may be extended to obtain that
	\begin{equation*}
  \tr \left(S\iint_{\Rdd}f \ dF\right) = \iint_{\Rdd} f \ d\mu_{S}^F
\end{equation*}
for each $f\in L^{\infty}(\Rdd)$, see the proof of lemma 6 in \cite{Kiukas:2006}. Furthermore, one can show \cite[Lemma 6(b)]{Kiukas:2006} that the covariance of $F$ implies that the measures $\mu_S^{F}$ are all absolutely continuous with respect to Lebesgue measure $\mu$. Hence $\mu_S^F$ has a Radon-Nikodym derivative $g_S\in L^1(\Rdd)$ such that $d\mu_S^F=g_S d\mu$. Using these facts we find that 
\begin{align*}
  \tr(S\Gamma(f))&= \tr \left(S\iint_{\Rdd}f_n \ dF\right)\\
  &= \iint_{\Rdd} f_n \ d\mu_{S}^F \\
  &=\iint_{\Rdd} f_n g_S d\mu \to \iint_{\Rdd} f g_S d\mu = \tr(S\Gamma(f))
\end{align*}
by the weak*-convergence of $\{f_n\}_{n\in \N}$.
 \end{enumerate}
\end{proof}
\begin{rem}
	If $S\in \tco$ is positive with $\tr(S)=1$, then $F(\Omega)=\chi_{\Omega}\star S$ for $\Omega \subset \Rdd$ defines a covariant POVM $F$ by proposition \ref{prop:genlocPOVM}, and $\iint_{\Rdd}f\ dF=f\star S$ for $f\in L^{\infty}(\Rdd)$ by proposition \ref{prop:convolutionasintegralofpovm}. The proof of proposition \ref{prop:covariantpovm} shows that $f\mapsto \iint_{\Rdd}f\ dF=f\star S$ satisfies the four axioms of theorem \ref{thm:poscorrule}, as we claimed in section \ref{sec:generalizedlocalization}.
\end{rem}
In our terminology this means that any covariant POVM $F$ is given by mixed-state localization operators: $F(\Omega)=\chi_{\Omega}\star S$ for some positive $S\in \tco$ with $\tr(S)=1$. In particular, the induced probability measures $\mu_{\psi}^F$ for $\psi\in \HS$ must be given by positive Cohen's class distributions with the correct total energy property, by lemma \ref{lem:radonnikodym}.

\end{document}